\documentclass[11pt,a4paper]{article}
\usepackage{enumerate}
\usepackage{amsmath}
\usepackage{fix-cm}
\usepackage{pb-diagram}
\usepackage[english]{babel}
\usepackage{amssymb,latexsym}
\usepackage{url}
\usepackage{atbegshi}
\usepackage{enumitem}  
\usepackage{tikz-cd}
\usepackage[dvipsnames]{xcolor}
\usepackage{authblk}
\newcommand{\thref}[1]{\hyperref[#1]{Theorem~\ref*{#1}}}
\newcommand{\corref}[1]{\hyperref[#1]{Corollary~\ref*{#1}}}
\newcommand{\propref}[1]{\hyperref[#1]{Proposition~\ref*{#1}}}
\newcommand{\secref}[1]{\hyperref[#1]{Section~\ref*{#1}}}
\newcommand{\lemref}[1]{\hyperref[#1]{Lemma~\ref*{#1}}}
\newcommand{\rkref}[1]{\hyperref[#1]{Remark~\ref*{#1}}}
\newcommand{\defref}[1]{\hyperref[#1]{Definition~\ref*{#1}}}
\newcommand{\notref}[1]{\hyperref[#1]{Notation~\ref*{#1}}}
\newcommand{\Conref}[1]{\hyperref[#1]{Conjecture~\ref*{#1}}}
\usepackage{amsthm}
\usepackage[pdftex, bookmarks, colorlinks=true, pdfpagemode=UseOutlines,
        linkcolor=blue,
        urlcolor=primaryblue,
        ]{hyperref}
\usepackage{hyperref}

\definecolor{primaryblue}{RGB}{70,130,180}

\usepackage{comment}

\DeclareMathOperator{\dive}{div}

\DeclareMathOperator{\ric}{Ric}

\DeclareMathOperator{\proj}{proj}

\DeclareMathOperator{\w}{\omega}

\DeclareMathOperator{\spann}{Span}
\DeclareMathOperator{\Abs}{{Abs,p}}

\def \bui#1#2{\mathrel{\mathop{\kern 0pt#1}\limits^{#2}}}
\def \buil#1#2{\mathrel{\mathop{\kern 0pt#1}\limits_{#2}}}

\let\<\langle
\let\>\rangle
\makeatletter
\renewcommand{\@cite}[2]{%
  \textcolor{green}{[}
  {#1\if@tempswa \textcolor{green}{, #2}\fi}
  \textcolor{green}{]}
}
\makeatother

\newtheorem{example}{Examples}[section]
\newtheorem{thm}{Theorem}[section]
\newtheorem{lemma}[thm]{Lemma}
\newtheorem{prop}[thm]{Proposition}

\newtheorem{cor}[thm]{Corollary}
\newtheorem{remark}[thm]{Remark}
\newtheorem{remarks}[thm]{Remarks}
\newtheorem{definition}[thm]{Definition}
\newtheorem{notation}[thm]{Notation}
\newtheorem{exabout:ample}[thm]{Example}
\newtheorem{conjecture}[thm]{Conjecture}

\begin{document}
\date{\normalsize \today}

\title{Geometric eigenvalue estimates of Kuttler-Sigillito type on differential forms}
\author[1,2]{Rodolphe Abou Assali\thanks{\texttt{rodolphe.abou-assali@univ-lorraine.fr, rodolphe.abouassali.1@ul.edu.lb}}}

\affil[1]{{\footnotesize Lebanese University, Faculty of Sciences II, Department of Mathematics, P.O. Box 90656 Fanar-Matn, Lebanon}}
\affil[2]{\footnotesize Universit\'e de Lorraine, CNRS, IECL, F-54000 Nancy, France}

\newcommand{\myAbst}[1]{
    \begin{center}
        \begin{minipage}{0.7\textwidth}
        \textbf{Abstract}. {\small #1}
        \end{minipage}
    \end{center}
}

\maketitle 
\myAbst{%
We introduce a new biharmonic Steklov problem on differential forms with Dirichlet-type boundary conditions and show that it is elliptic. We prove the existence of a discrete spectrum for this problem and give variational characterizations for eigenvalues associated to it. We establish eigenvalue estimates known as Kuttler-Sigillito inequalities, which connect the eigenvalues of different problems on differential forms with curvature quantities on the manifold.
}


\section{Introduction} \label{sec:introduction}
On planar domains, J. Kuttler and V. Sigillito \cite{Kuttler1, Kuttler&Sigillito} established in 1968 a series of inequalities relating the eigenvalues of the membrane problems (both Dirichlet and Neumann Laplacian eigenvalues), the biharmonic Steklov problems, and the classical Steklov problem. More recently, Hassannezhad and Siffert \cite{hassannezhadETsiffert} extended some of these results to compact Riemannian manifolds with $C^2$ boundary. In \cite{Article1}, we have lately extended to differential forms Kuttler-Sigillito inequalities that do not involve curvature.

Let $(M^n,g)$ be a compact Riemannian manifold with a smooth boundary and let $\nu$ be the inward unit normal along the boundary $\partial M$. Recall that the scalar biharmonic Steklov problem with Dirichlet boundary conditions DBS \cite{FerreroGazzolaWeth,hassannezhadETsiffert,Kuttler&Sigillito,RaulotSavo4} is defined as: 
\begin{equation*}
    \left\{
    \hspace{-0.3em}
    \begin{array}{r@{\quad}c@{\quad}l@{\quad}l}
        \Delta^2 f &=& 0 & \text{on } M \\[0.15em]
        f &=& 0 & \text{on } \partial M \\[0.15em]
        \Delta f - q \partial_\nu f &=& 0 & \text{on } \partial M
    \end{array}
    \right.
\end{equation*} and the scalar biharmonic Steklov problem with Neumann boundary conditions (NBS) \cite{hassannezhadETsiffert,Kuttler&Sigillito} is:
\begin{equation*}
    \left\{
    \hspace{-0.3em}
    \begin{array}{r@{\quad}c@{\quad}l@{\quad}l}
        \Delta^2 f &=& 0 & \text{on } M \\[0.15em]
        \partial_\nu f &=& 0 & \text{on } \partial M \\[0.15em]
        \partial_\nu (\Delta f) + \ell f &=& 0 & \text{on } \partial M
    \end{array}
    \right.
\end{equation*}
where $u$ is a smooth function on $M$, and $\Delta$ is the positive scalar Laplacian operator. In both problems, the real numbers $q$ and $\ell$ are the eigenvalues of the operator. 

 We denote by $\lambda_{k}$, $\mu_{k}$, $\sigma_{k}$, $q_{k}$ and $\ell_{k}$ the $k^{th}$ eigenvalues associated with scalar eigenfunctions for the Dirichlet, Neumann, Steklov, DBS and NBS problems, respectively
 
 We recall some of Kuttler-Sigillito inequalities. \begin{thm}\cite[Theorem 1.1]{hassannezhadETsiffert}\label{thm:firstineq}
    Let $(M^n, g)$ be a compact Riemannian manifold of dimension $n \geq 2$ with a $C^2$ boundary. For all integer $k\geq 1$, we have:
    \begin{equation*}
        \mu_k \sigma_1 \leq \ell_k \quad \text{and} \quad \mu_1 \sigma_k \leq \ell_k.
    \end{equation*}
\end{thm}
The theorem was proven in \cite{Kuttler&Sigillito} for $k=1$. On the other hand, we have the following:
\begin{thm}\cite[Table 1]{Kuttler&Sigillito}\label{thm:ineqKuttlerSigillito}
  On a bounded domain $M$ of the plane with piecewise $C^1$ boundary $\partial M$, the following inequalities hold:
\begin{center}
$\begin{array}{ll}
 & \bullet \quad q_1 \sigma_1^2 \leq \ell_1 \\[0.5em]
& \bullet \quad \mu_1^{-1} \leq \lambda_1^{-1} + (q_1 \ell_1)^{-\frac{1}{2}} \\[0.5em]
 & \bullet \quad \mu_1^{-1} \leq \lambda_1^{-1} + (q_1 \sigma_1)^{-1}
\end{array}$
\end{center}

\end{thm}

Let us fix $x_0\in M$ and assume that $M$ is star-shaped with respect to $x_0$, we define the following notation: $d_{x_0} : M \to [0, \infty);\ y \mapsto d_{x_0}(y) := d({x_0}, y)$ that is the Riemannian distance induced by the metric $g$, $\rho_{x_0}(y) := \frac{1}{2} d_{x_0}(y)^2,$ $r_{\max} := \underset{y \in M}{\max}\ d_{x_0}(y) = \underset{y \in \partial M}{\max}\ d_{x_0}(y)$. We also denote by $h:\partial M\longrightarrow \mathbb{R}$ the map given by \begin{equation}\label{eq:h}h := \langle -\nabla \rho_{x_0}, \nu \rangle,
\qquad
h_{\max} := \max_{y \in \partial M} \langle -\nabla \rho_{x_0}, \nu \rangle,
\qquad
h_{\min} := \min_{y \in \partial M} \langle -\nabla \rho_{x_0}, \nu \rangle.
\end{equation}
We recall the definition of a star-shaped manifold: 
\begin{itemize}
  \item A domain \(U\subset \mathbb{R}^n\) is said to be star-shaped with respect to a point \(x_0\) if for every \(x\in U\), the line segment connecting \(x_0\) to \(x\) is entirely contained in \(U\).
\item A manifold $M$ with $C^2$ boundary is said to be \emph{star-shaped} with respect to $x_0 \in M$ if there exists a star-shaped domain 
$\Omega \subset \mathbb{R}^n \simeq T_{x_0}M$ with respect to $0$, such that the exponential map $\exp_{x_0}$ is defined on $\Omega$ and $\exp_{x_0}:\Omega \to M$ is a diffeomorphism.  

\end{itemize}

The star--shaped property ensures that $\langle -\nabla \rho_{x_0}(x), \nu(x) \rangle > 0$ for all $x \in \partial M$. In particular, $h_{\max}\geq h_{\min}>0.$

 \begin{thm}\label{thm:ineqScalaire1}\cite[Theorem 1.3]{hassannezhadETsiffert}
     Let $(M^n,g)$ be a compact Riemannian manifold whose Ricci curvature $\ric_g$ satisfies $\ric_g\geq(n-1)\kappa$, for $\kappa\in\mathbb{R}$. Assume that $M$ is star-shaped with respect to $x_0\in M$, then we have
     \begin{equation*}
         \sigma_1\geq \frac{h_{min}\mu_1}{2r_{max}\mu_1^{\frac{1}{2}}+C_0}
     \end{equation*} where $C_0:=C_0(n,\kappa,r_{max})$ is a positive constant depending only on $n,\ \kappa$ and $r_{max}.$
 \end{thm}
It was shown in \cite[Inequality (VIII)]{Kuttler&Sigillito} that when $M$ is a domain in $\mathbb{R}^2$, the constant reduces to $C_0=2$. This inequality was later generalized to all eigenvalues in \cite[Theorem 2.3]{Berge}. The following theorem generalizes the inequalities (V) and (VI) of \cite{Kuttler&Sigillito}.

\begin{thm}\label{thm:ineqScalaire2}\cite[Theorem 1.4]{hassannezhadETsiffert}
    Let $(M^n, g)$ be a compact Riemannian manifold whose sectional curvature $\kappa_g$ satisfies $\kappa_1 \le \kappa_g \le \kappa_2$. Moreover, assume that there exists ${x_0} \in M$ such that $M$ is star-shaped with respect to ${x_0}$ and the cut locus of ${x_0}$ in $M$ is the empty set. Then there exist positive constants $C_i := C_i(n, \kappa_1, \kappa_2, r_{\max})$, $i=1,2$ such that
\begin{center}
$\begin{array}{l}
\bullet \quad \lambda_k \leq \frac{4r^2_{\max} q^2_k - 2c_1 h_{\min} q_k}{h_{\min}} \\[1em]
\bullet \quad \lambda_k \geq \frac{c_2 q_m }{h_{\max}}
\end{array}$
\end{center}
Here, $m$ is the multiplicity of $\lambda_k$.
\end{thm}

 In the same spirit of \cite{Article1} where we have extended \thref{thm:firstineq} and \thref{thm:ineqKuttlerSigillito} to differential forms, we establish in this paper estimates that generalize the curvature-dependent inequalities of \thref{thm:ineqScalaire1} and \thref{thm:ineqScalaire2}.

We denote by $d$ the exterior differential operator and by $\delta$ its $L^2-$adjoint. We also let $\iota: \partial M \to M$ the inclusion map, and $\lrcorner$ the interior product (see \secref{sec:diffForms}). Now, we recall the biharmonic Steklov problem with Dirichlet-type conditions DBS which has been extended to differential forms by F. El Chami, N. Ginoux, G. Habib, and O. Makhoul in \cite{BSFidaGeorgeOlaNicolas}. For a compact Riemannian manifold $(M^n,g)$ with smooth boundary $\partial M$, the boundary problem \begin{equation}\label{eq:DBSF1}
(DBS)\begin{cases}
    \Delta^2 \w = 0 & \text{on }\ M \\
     \w = 0 &  \text{on  } \partial M\\
    \nu \lrcorner \Delta \w + q \iota^*  \delta\w = 0 & \text{on  } \partial M \\
     \iota^*  \Delta \w- q\nu \lrcorner d \w   = 0 & \text{on  } \partial M
\end{cases}
\end{equation} on $p$-forms, $p\in\{0,\ldots,n\}$ has a discrete spectrum consisting of an unbounded monotonously nondecreasing sequence of positive eigenvalues of finite multiplicities $(q_{j,p})_{j\geq1}.$

\begin{thm}\cite[Theorem 2.6]{BSFidaGeorgeOlaNicolas}
    The first eigenvalue $q_{1,p}$ of \eqref{eq:DBSF1} can be characterized as follows:
\begin{eqnarray}
\label{eq:q_1Standard}q_{1,p}&=&\inf\left\{\frac{\|\Delta\omega\|_{L^2(M)}^2}{
\|\nu\lrcorner d\omega\|_{L^2(\partial
M)}^2+\|\iota^*\delta\omega\|_{L^2(\partial
M)}^2}\,|\,\omega\in\Omega^p(M),\,\omega_{|_{\partial
M}}=0\ \text{and}\ \nabla_\nu\w\neq 0\right\}\\
\label{eq:q_1Alternative}&=&\inf\left\{\frac{\|\omega\|_{L^2(\partial M)}^2}{
\|\omega\|_{L^2(
M)}^2}\,|\,\omega\in\Omega^p(M)\setminus\{0\},\,\;\Delta\omega=0\textrm{ on
}M\right\}.
\end{eqnarray}
\end{thm} Both infima are indeed minima: \eqref{eq:q_1Standard} is attained by an eigenform of \eqref{eq:DBSF1}, associated to $q_{1,p}$ and \eqref{eq:q_1Alternative}
is attained by $\Delta\w$, where $\w$ is an eigenform of \eqref{eq:DBSF1}, associated to $q_{1,p}.$

In \secref{sec:DBS2} of this paper, we introduce for $p\in\{0,\cdots,n-1\}$ a new DBS problem as follows \begin{equation*}
    (DBS2)
    \left\{
    \hspace{-0.3em}
    \begin{array}{r@{\quad}c@{\quad}l@{\quad}l}
        \Delta^2 \w &=& 0 & \text{on } M \\[0.15em]
        \w &=& 0 & \text{on } \partial M \\[0.15em]
        \iota^* \delta\w &=& 0 & \text{on } \partial M \\[0.15em]
        \iota^* \Delta \w - \mathbf{q}\nu \lrcorner d\w &=& 0 & \text{on } \partial M
    \end{array}
    \right.
\end{equation*} that we call DBS2. Actually, for $p=0$, it is the scalar DBS and for $p=n$, it is trivial (such an $\omega$ must vanish on $M$).

Of particular note, the (NBS) problem for $p$-differential forms is generalized in \cite{Article1}. We have studied three different NBS problems, each of which has a discrete spectrum consisting of an unbounded, non-decreasing sequence of positive real eigenvalues with finite multiplicities. These problems (NBS),(NBS2) and (NBS3) respectively are as follows: 
{\small\begin{equation*}
\begin{array}{cc}
  \begin{array}[t]{r@{\quad}c@{\quad}l@{\quad}l}
      \left\{
      \begin{array}{r@{\quad}c@{\quad}l@{\quad}l}
          \Delta^2 \w &=& 0 & \text{on } M \\[0.15em]
          \nu \lrcorner \w &=& 0 & \text{on } \partial M \\[0.15em]
          \nu \lrcorner d \w &=& 0 & \text{on } \partial M \\[0.15em]
          \nu \lrcorner \Delta \w &=& 0 & \text{on } \partial M \\[0.15em]
          \nu \lrcorner \Delta d\w + \ell \iota^* \w &=& 0 & \text{on } \partial M
      \end{array}
      \right.
  \end{array}
  & 
  \begin{array}[t]{r@{\quad}c@{\quad}l@{\quad}l}
      \left\{
      \begin{array}{r@{\quad}c@{\quad}l@{\quad}l}
          \Delta^2 \w &=& 0 & \text{on } M \\[0.15em]
          \nu \lrcorner \w &=& 0 & \text{on } \partial M \\[0.15em]
          \nu \lrcorner d\w &=& 0 & \text{on } \partial M \\[0.15em]
          \iota^* \delta \w &=& 0 & \text{on } \partial M \\[0.15em]
          \nu \lrcorner \Delta d \w + \ell^{NBS2} \iota^* \w &=& 0 & \text{on } \partial M
      \end{array}
      \right.
  \end{array}
  \\[2cm]
  \multicolumn{2}{c}{%
    \begin{array}[t]{r@{\quad}c@{\quad}l@{\quad}l}
        \left\{
        \begin{array}{r@{\quad}c@{\quad}l@{\quad}l}
            \Delta^2 \w &=& 0 & \text{on } M \\[0.15em]
            \nu \lrcorner \w &=& 0 & \text{on } \partial M \\[0.15em]
            \nu \lrcorner d\w &=& 0 & \text{on } \partial M \\[0.15em]
            \nu \lrcorner \Delta \w + \ell^{NBS3} \iota^* \delta \w &=& 0 & \text{on } \partial M \\[0.15em]
            \nu \lrcorner \Delta d\w + \ell^{NBS3} \iota^* \w &=& 0 & \text{on } \partial M.
        \end{array}
        \right.
    \end{array}%
  }
\end{array}
\end{equation*}}
We now recall the Bochner formula:
 \begin{thm}\cite[Paragraph 2]{RaulotSavo2}\label{thm:bochner}
 On differential forms, we have \[\Delta=\nabla^*\nabla+W^{[p]}\] where $W^{[p]}$ is the curvature term, which is a self-adjoint endomorphism acting on $p$-forms.
 \end{thm}
 We also recall that the shape operator $S^{[p]}$:
 \begin{definition}\cite[Paragraph 2]{RaulotSavo2}\label{def:shape}
   The shape operator $S$ acting on $T(\partial M)$ is defined by $S(X)=-\nabla_X\nu$ for all $x\in \partial M$, $X\in T_x\partial M$. It is extended to all $\alpha\in\Omega^p(\partial M)$ by \[(S^{[p]}\alpha) (X_1,\ldots,X_p)=\sum_{i=1}^p\alpha (X_1,\ldots,S(X_i),\ldots,X_p).\]
\end{definition}
 We denote by $\lambda_{k,p}$, $\mu_{k,p}$, $\sigma_{k,p}$, $q_{k,p}$ and $\mathbf{q}_{k,p}$ the $k^{th}$ eigenvalues for differential forms for arbitrary degree $p$ of the following problems, respectively: Dirichlet, Neumann, Steklov, DBS and DBS2.
 
The following generalizations of \thref{thm:ineqScalaire1} and \thref{thm:ineqScalaire2} are the main results of this note. Recall that $h_{\min}$ and $h_{\max}$ have been defined in \eqref{eq:h}.
     \begin{thm}\label{thm:ineg1}
         
     Let $(M^n,g)$ be a compact Riemannian manifold with a smooth boundary $\partial M$. Assume that $M$ is star-shaped with respect to $x_0$ and that its Ricci curvature satisfies $\ric_g\geq (n-1)\kappa$ where $\kappa\in\mathbb{R}$, $W^{[p]}\geq 0$, and $S^{[p]}\geq 0$. Then for $p\geq1$, we have  
     \begin{equation}\label{eq:sigma1>mu1/mu_1}
      \sigma_{1,p}> \frac{\frac{1}{2}h_{\min} \mu_{1,p}}{r_{\max}\mu_{1,p}^{\frac{1}{2}}+\frac{1}{2}\max(1+d_{x_0}.s_\kappa\circ d_{x_0})}.
     \end{equation}
     \end{thm}
    
     \begin{thm}\label{thm:ineg2}Let $(M^n,g)$ be a compact Riemannian manifold with a smooth boundary $\partial M$. Assume that $M$ is star-shaped with respect to $x_0$ and that $\kappa_1\leq \kappa_g \leq \kappa_2$ where $\kappa_g$ is the sectional curvature, then for all $k\in \mathbb{N}$, $p\geq 1$, ${x_0} \in M$ there exists a constant $c_1=c_1(p,n,\kappa_1,\kappa_2)$ such that:  

\begin{equation}\label{eq:th1.4(i,2)}
     \lambda_{k,p}\leq \frac{4\mathbf{q}^2_{k,p}r^2_{\max}+2\mathbf{q}_{k,p}h_{\min} c_1}{h^2_{\min}}.
\end{equation}
\end{thm}
Extending the proof of the second inequality of \thref{thm:ineqScalaire2} yields a constant $c_2$ that can sometimes be negative for $p\geq 1$. However, the expected result would be as follows:
\begin{conjecture}
    Let $(M^n,g)$ be a compact Riemannian manifold, whose sectional curvature $\kappa_g$ satisfies $\kappa_1\leq \kappa_g \leq \kappa_2$. Assume that $M$ is star-shaped with respect to $x_0\in M$ then, for all $k\in\mathbb{N}$, $p\geq 1$, there exists a constant $c_2= c_2(p,n,\kappa_1,\kappa_2)>0$ such that:

\begin{equation}\label{eq:th1.4(i,1)}
 \qquad    \lambda_{k,p}\geq \frac{c_2 q_{{m_k},p}}{r_{\max}+h_{\max}} . 
\end{equation}
Here $m_k$ is the multiplicity of $\lambda_{k,p}$.
\end{conjecture}
This inequality could be of interest when $m$ is sufficiently large compared with $k.$ To show inequality \eqref{eq:sigma1>mu1/mu_1}, we need the variational characterizations for the eigenvalues of the Steklov and Neumann problems on differential forms, see \secref{sec:boundarypbs}. For inequality \eqref{eq:th1.4(i,2)}, we use the variational characterizations of the Dirichlet problem (see \secref{sec:boundarypbs}) and the new biharmonic Steklov problem with Dirichlet boundary conditions (DBS2) since it is not clear whether an analogue of \eqref{eq:th1.4(i,2)} holds for (DBS).

\begin{remark}
     We studied the NBS spectrum and are currently studying the spectrum of the Neumann and DBS2 problems on differential forms on the unit ball, in order to check how far the inequalities established both in this paper (\thref{thm:ineg1} and \thref{thm:ineg2}) and in \cite{Article1} are from being optimal.
\end{remark}

\textbf{Outline of the Article:} In \secref{sec:Importantthandtools} of this paper, we recall some definitions and ingredients that we use, such as integration by parts formulas and comparison theorems. In \secref{sec:Rellich} we establish the Rellich identity for differential forms. In \secref{sec:DBS2} we introduce the new biharmonic Steklov problem with Dirichlet boundary conditions on differential forms. Finally, in \secref{sec:eigenvalueEstimates}, we show \thref{thm:ineg1} and \thref{thm:ineg2}.

\textbf{Acknowledgment:} I would like to thank my PhD supervisors, Nicolas Ginoux, Georges Habib and Samuel Tapie, for their help, support, and the valuable discussions we have had. I am also thankful to Fida El Chami and Georges Habib for providing me with the generalized Rellich identity for differential forms \eqref{Rellich_forms} (an important tool to establish the eigenvalue estimates we want), which they had previously developed in an unpublished work \cite{RellichFormesDiff}. I also acknowledge the support of the Grant ANR-24-CE40-0702 ORBISCAR.

\section{Preliminaries}\label{sec:Importantthandtools}
In this section, we provide some results and tools that are used in our work. The results of \secref{sec:diffForms} and \secref{sec:boundarypbs} are already in \cite{Article1}.

\subsection{Review on differential forms}\label{sec:diffForms}
    For differential forms, we denote by $\langle.,.\rangle$ the natural inner product induced by the Riemannian metric and by $d\mu_g$ the Riemannian volume form. We define $(\w,\w'):=\int_M\langle\w,\w'\rangle d\mu_g.$
    The codifferential operator $\delta$, is the formal adjoint of the exterior derivative $d$, i.e., for all $\alpha\in \Omega^p(M)$ and $\beta\in \Omega^{p+1}(M)$ with support in $M\setminus \partial M$, we have  $(d\alpha, \beta) = (\alpha, \delta\beta).$ Throughout this paper, we identify the tangent space $TM$ with its cotangent space $T^*M$ via the musical isomorphism. That is, for a vector field $X$, we denote by $X^\flat = g(X, \cdot)$ the corresponding $1$-form. Given a local orthonormal basis  $(e_i)_{i=1}^n$, we have the local expressions of $d=\sum_i e_i^\flat\wedge\nabla_{e_i}$ and $\delta=-\sum_ie_i\lrcorner\nabla_{e_i}$, see for instance \cite[Proposition 2.61]{GallotHulinLafontaine}. The Hodge-de-Rham Laplacian is the operator $\Delta$ acting on $\Omega^p(M)$, defined by
    \[\Delta  = d\delta + \delta d.\] 
We recall the following integration by parts formulas that are essential in our work: \begin{prop}\label{prop:IPP}
    For $\w \in \Omega^p(M)$ and $\w' \in \Omega^{p+1}(M)$, we have 
    \begin{equation}\label{eq:IPP}
        \int_M \langle d\w, \w' \rangle d\mu_g = \int_M \langle \w, \delta\w' \rangle d\mu_g -\int_{\partial M} \langle \iota^* \w, \nu \lrcorner \w' \rangle d\mu_g.
    \end{equation}
For all $\w,\w'\in \Omega^p(M)$ we have
    \begin{equation}\label{eq:IPP3'}
    \begin{aligned}
    \int_M \langle \Delta \w, \w' \rangle \, d\mu_g &= \int_M \langle d\w, d\w' \rangle \, d\mu_g + \int_M \langle \delta \w, \delta \w' \rangle \, d\mu_g \\
    &\quad + \int_{\partial M} \langle \nu \lrcorner d\w, \iota^* \w' \rangle \, d\mu_g - \int_{\partial M} \langle \iota^* \delta \w, \nu \lrcorner \w' \rangle \, d\mu_g.
    \end{aligned}
    \end{equation}
    \end{prop}
    Identity \eqref{eq:IPP} is shown in \cite[p.182]{taylor}. It implies \eqref{eq:IPP3'}, as shown in \cite[Equation $30$]{BesselFidaGeorgeNicolas} and \cite[p.421]{taylor}.
   In particular, for any $\w, \w' \in \Omega^p(M)$, we deduce \cite[p.182]{taylor}
    \begin{equation}\label{eq:IPP1}
    \begin{split}
        & \int_M \Big( \langle \Delta \w , \w' \rangle -  \langle \w , \Delta \w' \rangle \Big) d\mu_g \\
        = & \int_{\partial M} \Big( \langle \nu \lrcorner d \w, \iota^* \w' \rangle - \langle \iota^* \w, \nu \lrcorner d \w' \rangle + \langle \nu \lrcorner \w, \iota^* \delta \w' \rangle - \langle \iota^* \delta \w, \nu \lrcorner \w' \rangle \Big) d\mu_g.
    \end{split}
    \end{equation}

Replacing $\w$ with $\Delta \w$ in \eqref{eq:IPP1}, we obtain:
    \begin{equation}\label{eq:IPP2}
    \begin{split}
        & \int_M \langle \Delta^2 \w , \w' \rangle d\mu_g = \int_M \langle \Delta \w , \Delta \w' \rangle d\mu_g \\
        & + \int_{\partial M} \Big( \langle \nu \lrcorner d \Delta \w, \iota^* \w' \rangle - \langle \iota^* \Delta \w, \nu \lrcorner d \w' \rangle + \langle \nu \lrcorner \Delta \w, \iota^* \delta \w' \rangle - \langle \iota^* \delta \Delta \w, \nu \lrcorner \w' \rangle \Big) d\mu_g.
    \end{split}
    \end{equation}
        In addition, for $\w = \w'$ in \eqref{eq:IPP3'}, we deduce the following formula:
    \begin{equation}\label{eq:IPP3}
    \begin{aligned}
       \int_M \langle \Delta \w, \w \rangle \, d\mu_g &= \int_M \left( \lvert d\w \rvert^2 + \lvert \delta\w \rvert^2 \right) d\mu_g + \int_{\partial M} \langle \nu \lrcorner d\w, \iota^* \w \rangle \, d\mu_g \\
       &\quad - \int_{\partial M} \langle \iota^* \delta \w , \nu \lrcorner \w \rangle \, d\mu_g.
    \end{aligned}
    \end{equation}
Note that, in \cite{taylor}, the identities are given using $ \w $ on the boundary rather than  $\iota^*\w$, as we do in this paper. This difference is because the normal component is always zero in our three NBS problems.

\subsection{Boundary problems and variational characterizations}\label{sec:boundarypbs}
In this section, we review the classical problems on differential forms and the variational characterizations of their eigenvalues, which are derived from the min-max principle (see \cite[Chapter 1]{Chavel}), as we use them to establish the Kuttler-Sigillito inequalities. The proofs primarily follow the same lines as those used for the scalar case and we omit them here. In this context, we let $(M^n,g)$ be a compact connected Riemannian manifold with boundary $\partial M$. The Dirichlet problem for differential forms is defined as follows:
\begin{equation}\label{eq:dirichletFormes}
    \left\{
    \hspace{-0.3em}
    \begin{array}{r@{\quad}c@{\quad}l@{\quad}l}
        \Delta \w &=& \lambda \w & \text{on } M \\[0.5em]
        \w &=& 0 & \text{on } \partial M
    \end{array}
    \right.
\end{equation}
where $\lambda \in \mathbb{R}$. For $p\in\{0,\cdots,n\}$, the kernel of the Dirichlet problem is trivial (see, \cite[Chapter 5]{taylor}). The spectrum of the Laplacian of Dirichlet on forms is discrete, consisting of eigenvalues 
\[0<\lambda_{1,p} \leq \lambda_{2,p} \leq \ldots\]
The $k\textsuperscript{th}$ eigenvalue $\lambda_{k,p}$ \cite{BesselFidaGeorgeNicolas,GueriniThese} is given by \begin{align}
\lambda_{k,p}&=\underset{\substack{V\subset{H}^1_0(M)\\ \dim(V)=k}}{\min}\ \underset{\substack{0\neq \w \in V }}{\max} \frac{\lVert d\w \rVert^2_{L^2(M)} + \lVert \delta \w \rVert^2_{L^2(M)}}{\lVert \w \rVert^2_{L^2( M)}}\label{eq:vpKDirichletForme} \\ 
&=\underset{\substack{V\subset H_0^1(M)\cap H^2(M)\\ \dim(V)=k}}{\min}\ \underset{\substack{0\neq \nabla\w \in V }}{\max} \frac{\lVert \Delta\w \rVert^2_{L^2( M)}}{\lVert d\w \rVert^2_{L^2(M)} + \lVert \delta \w \rVert^2_{L^2(M)}}.\label{eq:vpKAlternativeDirichletForme}
\end{align}
We recall the absolute de Rham cohomology \begin{equation}\label{eq:Hap}
        H_{\Abs}(M) = \left\{ \w \in \Omega^p(M) \mid d\w = 0 \text{ on } M, \ \delta\w = 0 \text{ on } M, \ \nu \lrcorner \w = 0 \text{ on } \partial M \right\}.
\end{equation}The Neumann problem for differential forms (with absolute boundary conditions) is defined as follows:
\begin{equation}\label{eq:neumannFormesAbsolue}
    \left\{
    \hspace{-0.3em}
    \begin{array}{r@{\quad}c@{\quad}l@{\quad}l}
        \Delta \w &=& \mu \w & \text{on } M \\[0.5em]
        \nu \lrcorner \w &=& 0 & \text{on } \partial M \\[0.5em]
        \nu \lrcorner d\w &=& 0 & \text{on } \partial M
    \end{array}
    \right.
\end{equation}
where $\mu\in \mathbb{R}$. The spectrum of the Laplacian of Neumann on forms is discrete, consisting of eigenvalues 
\[0<\mu_{1,p} \leq \mu_{2,p} \leq \ldots \]Note that $\mu_{0,p}=0$ is an eigenvalue if and only if $H_{\Abs}(M)\neq0.$ For $p\in\{0,\cdots,n\}$, using the integration by parts formula \eqref{eq:IPP3}, we deduce that the kernel of the Neumann problems is $H_{\Abs}(M)$ (see \eqref{eq:Hap} for its definition). We define the following space \begin{align}
    H_N^1(M) &:= \left\{ \w \in H^1(M) \mid \nu \lrcorner \w = 0 \text{ on } \partial M \right\}.
\end{align}
 The first non-zero eigenvalue $\mu_{1,p}$ \cite{NeumannProblemConner,BesselFidaGeorgeNicolas,GueriniThese} of the problem \eqref{eq:neumannFormesAbsolue} is given by
\begin{align}\label{eq:vpNeumannAbsolueForme}
\mu_{1,p} &= \inf\left\{\frac{\lVert d\w \rVert^2_{L^2(M)} + \lVert \delta \w \rVert^2_{L^2(M)}}{\lVert \w \rVert^2_{L^2( M)}}\ \Big|\ \w\in H_N^1(M)\ \text{and}\ \w\perp_{L^2(M)}H_{\Abs}(M)\right\}.
\end{align}
 Its $k\textsuperscript{th}$ non-zero eigenvalue $\mu_{k,p}$ is given by 
   \begin{align}
      \mu_{k,p}= \underset{\substack{ V\subset H_N^1(M)\\ \dim(V)=k+\dim H_{\Abs}(M)}}{\inf}\ \underset{\substack{0\neq \w \in V }}{\sup} \frac{\lVert d\w \rVert^2_{L^2(M)} + \lVert \delta \w \rVert^2_{L^2(M)}}{\lVert \w \rVert^2_{L^2( M)}}\label{eq:vp1AlternativeNeumannAbsolueForme1}. 
   \end{align}
The Steklov (Dirichlet-to-Neumann) problem for differential forms is defined as follows: \cite{RaulotSavo}
\begin{equation}\label{eq:steklovFormes}
    \left\{
    \hspace{-0.3em}
    \begin{array}{r@{\quad}c@{\quad}l@{\quad}l}
        \Delta {\w} &=& 0 & \text{on } M \\[0.15em]
        \nu \lrcorner {\w} &=& 0 & \text{on } \partial M \\[0.15em]
        \nu \lrcorner d\w &=& -\sigma \iota^* {\w} & \text{on } \partial M
    \end{array}
    \right.
\end{equation}where $\sigma\in \mathbb{R}$. The spectrum of the Steklov operator on forms is discrete, consisting of eigenvalues 
\[0 < \sigma_{1,p} \leq \sigma_{2,p} \leq \ldots 
\]Note that $\sigma_{0,p}=0$ is an eigenvalue if and only if $H_{\Abs}(M)\neq0.$  For $p\in\{0,\cdots,n\}$, the kernel of the Steklov problem is $H_{\Abs}(M)$ (see \eqref{eq:Hap} for its definition).  The first non-zero eigenvalue of the Steklov problem on differential forms is given by
   \begin{align}\label{eq:sigma1formesSteklov}
\sigma_{1,p} &=\inf\left\{\frac{\lVert d\w \rVert^2_{L^2(M)} + \lVert \delta \w \rVert^2_{L^2(M)}}{\lVert \w \rVert^2_{L^2( \partial M)}}\ \Big|\ {\w}\in H_N^1(M)\ \text{and}\ \w\perp_{L^2(\partial M)}H_{\Abs}(M)\right\}. 
\end{align}

\subsection{Comparison theorems}
Below are the comparison theorems for the Laplacian and the Hessian that will be used to prove the inequalities stated in \thref{thm:ineg1} and \thref{thm:ineg2} to provide lower bounds for geometric terms.
\begin{definition}\cite[Section 2.2]{hassannezhadETsiffert}
For all $\kappa \in \mathbb{R},$ we denote by $s_{\kappa}:[0,\infty)\longrightarrow \mathbb{R}$ the function that satisfies the following Riccati differential equation
\[s'_{\kappa}+s^2_{\kappa}+\kappa=0,\ \text{with}\ \lim_{r\to 0} \frac{rs_{\kappa}(r)}{n-1}=1.\]\end{definition}

Actually, we have
\[s_{\kappa}(r) =
\begin{cases}
  (n-1)\sqrt{\kappa}\cot(\sqrt{\kappa}r) & \text{if}\ \kappa>0 \\
  \frac{n-1}{r} & \text{if}\ \kappa=0\\
 (n-1)\sqrt{|\kappa|}\coth(\sqrt{|\kappa|}r) & \text{if}\ \kappa<0. \\
\end{cases}\]

We denote by $A\wedge B:=\min\{A,B\}\ \text{and}\ A\vee B:=\max\{A,B\},$ for $A,B \in \mathbb{R}.$ The following theorems are stated in \cite [Theorems 1.59 and 1.82]{ChowLuNi} and \cite[Theorems 2.1 and 2.2]{hassannezhadETsiffert}.

\begin{thm}\label{thm:hess} 
Let $(M^n,g)$ be a complete Riemannian manifold. Let $\gamma:[0,L]\rightarrow M$ be a minimal geodesic from a point ${x_0}$ on $M$, such that the intersection of its image with the cut-locus of $x_0$ is empty. Suppose also that $\kappa_1\leq \kappa_g(X,\gamma'(t))\leq \kappa_2$ for all $t\in [0,L]$ and $X\in T_{\gamma(t)}M$ orthogonal to $\gamma'(t)$. Then:
\begin{enumerate}
    \item[$\bullet$]  The function $ d_{x_0} $ satisfies the following inequalities:
    \begin{align*}
        \nabla^2d_{x_0}(X,X) &\leq \frac{s_{\kappa_1}(t)}{n-1}g(X,X), 
        && \forall t\in [0,L],\ X\in \langle \gamma'(t)\rangle^{\perp} \subset T_{\gamma(t)}M. \\
        \nabla^2d_{x_0}(X,X) &\geq \frac{s_{\kappa_2}(t)}{n-1}g(X,X), 
        && \forall t\in \left[0, L\wedge\frac{\pi}{2\sqrt{\kappa_2\vee0}}\right],\ X\in \langle \gamma'(t)\rangle^{\perp} \subset T_{\gamma(t)}M.
    \end{align*}
    \item[$\bullet$]  The function $ \rho_{x_0} $ satisfies the following estimates:
    \begin{align*}
        \nabla^2\rho_{x_0}(X,X) &\leq \frac{ts_{\kappa_1}(t)}{n-1}g(X,X), 
        && \forall t\in [0,L],\ X\in \langle \gamma'(t)\rangle^{\perp} \subset T_{\gamma(t)}M. \\
        \nabla^2\rho_{x_0}(X,X) &\geq \frac{ts_{\kappa_2}(t)}{n-1}g(X,X), 
        && \forall t\in [0,L\wedge\frac{\pi}{2\sqrt{\kappa_2\vee0}}],\ X\in \langle \gamma'(t)\rangle^{\perp} \subset T_{\gamma(t)}M.
    \end{align*}
\end{enumerate}
\end{thm}

\begin{thm}\label{thm:laplace}
Let $(M^n,g)$ be a complete Riemannian manifold. The distance function and the squared distance function satisfy the following inequalities:

\begin{enumerate}
    \item[$\bullet$]  If $ \ric_g \geq (n-1)\kappa, \ \kappa \in \mathbb{R},$ then, for all $ {x_0} \in M $, we have 
    \[
    -\Delta d_{x_0}(y) \leq s_{\kappa}(d_{x_0}(y)) \quad \text{and} \quad -\Delta \rho_{x_0}(y) \leq 1 + d_{x_0}(y). s_{\kappa}(d_{x_0}(y)).
    \]
    
    \item[$\bullet$]  Under the same assumptions as in \thref{thm:hess}, we have:
    \begin{enumerate}
        \item[$*$] For all $ t \in [0,L] $,
        \[
        -\Delta d_{x_0}(\gamma(t)) \leq s_{\kappa_1}(t) \quad \text{and} \quad -\Delta \rho_{x_0}(\gamma(t)) \leq 1 + t s_{\kappa_1}(t).
        \]
        
        \item[$*$] For all $ t \in \left[0, L \wedge \frac{\pi}{2\sqrt{\kappa_2 \vee 0}}\right] $, we have
        \[
        -\Delta d_{x_0}(\gamma(t)) \geq s_{\kappa_2}(t) \quad \text{and} \quad -\Delta \rho_{x_0}(\gamma(t)) \geq 1 + t s_{\kappa_2}(t).
        \]
    \end{enumerate}
\end{enumerate}

\end{thm}

\section{Rellich identity for differential forms}\label{sec:Rellich} 
In this section we give the generalization of the Rellich identity for differential forms \cite{RellichFormesDiff}, which is a particular expression of integration by parts that we use to generalize the first inequality of \thref{thm:ineqScalaire2}. This identity was first introduced and used by Rellich \cite{Rellich} to study Dirichlet eigenvalue problems when $M\subset\mathbb{R}^n$. Subsequently, this result was extended to Neumann, clamped plate, and buckling eigenvalue problems in \cite{Liu3}. The general Rellich identity for the Laplacian, as well as a higher order version for the bi-Laplacian, can be found in \cite{hassannezhadETsiffert}. For domains $M\subset \mathbb{R}^n$, we refer to \cite{Mitidieri} for the higher order identity. This identity is also referred to as the Pohozaev identity, see for instance, \cite{ColboisGirouardHassanezhad, ProvenzanoStubbe, Xiong} for cases where $f$ is harmonic, where it is used to get inequalities between the Steklov and Laplace eigenvalues. 

Let $(M^n,g)$ be a compact manifold with smooth boundary $\partial M$. Given a
smooth vector field $F$ on $M$, we define a self-adjoint endomorphism
$T_F^{[p]}$ of $\Omega^p(M)$ as follows (see \cite[Section 2.1 - Equation 2.2]{SavoRellich}): if $\omega$ is a $p$-form
and $X_1, \dots , X_p$ are tangent vector fields, then
$$(T_F^{[p]}\w) (X_1, \dots , X_p)=\sum_{k=1}^p \omega (X_1, \dots,\nabla_{X_k} F,  \dots, X_p). $$
\begin{lemma} \cite[Lemma 2.1]{SavoRellich}
  Let $\omega$ be a $p$-form and $F$ a vector field on $M$. We have:
  $$\mathcal{L}_F \omega = \nabla_F \omega + T_F^{[p]}\omega$$
  where $\mathcal{L}_F$ is the Lie derivative in the direction of $F$.
\end{lemma}
We now state the Rellich (or Pohozaev) identity on differential forms. This formula expresses the Laplacian of a differential form in terms of expressions of order at most $1.$ 
\begin{thm}\label{theo:Rellich}\cite{RellichFormesDiff} Let $(M^n,g)$ be a compact Riemannian manifold with smooth boundary $\partial M$ and let $F : M \rightarrow TM$ be a Lipschitz vector field on $M$. Then for every $\omega \in \Omega^p(M)$ we have
 \begin{align}\label{Rellich_forms} 
&\int_{M}\langle \Delta \omega , F\lrcorner d \omega \rangle \, d\mu_g 
 +\int_M \langle \delta \omega , F\lrcorner \Delta \omega \rangle \, d\mu_g \nonumber \\[1ex]
&\quad = -\dfrac{1}{2}\int_{\partial M} (|d \omega|^2 +|\delta \omega|^2) \langle F,\nu \rangle \, d\mu_g 
 + \int_{\partial M} \langle F \wedge \iota^* (\delta \omega), \nu \lrcorner d \omega \rangle \, d\mu_g \nonumber \\
&\quad\quad + \int_{\partial M} \langle \iota^* (F\lrcorner d \omega), \nu \lrcorner d \omega \rangle \, d\mu_g 
 + \int_{\partial M} \langle \iota^* (F\lrcorner \delta \omega), \nu \lrcorner \delta \omega \rangle \, d\mu_g \nonumber\\[1ex]
&\quad\quad - \dfrac{1}{2}\int_{M} (|d \omega|^2 +|\delta \omega|^2) \dive\, (F) \, d\mu_g
 + \int_{M} \langle \delta \omega, dF \lrcorner d \omega \rangle \, d\mu_g \nonumber \\
&\quad\quad + \int_{M} \langle T_F^{[p+1]} d \omega, d \omega \rangle \, d\mu_g
 + \int_{M} \langle T_F^{[p-1]} \delta \omega, \delta \omega \rangle \, d\mu_g.
\end{align}
\end{thm}

\begin{proof}
The proof is essentially based on the integration by parts formulas stated in \propref{prop:IPP}. We compute
\begin{eqnarray}
\label{Rellich_calcul}
 \int_{M} \langle \Delta \omega , F\lrcorner d \omega \rangle d\mu_g &=&  \int_{M} \langle \delta d \omega , F\lrcorner d \omega \rangle d\mu_g +  \int_{M} \langle d \delta \omega , F\lrcorner d \omega \rangle d\mu_g \nonumber \\
 &=& \int_{M}  \langle d \omega, d(F \lrcorner d \omega) \rangle  d\mu_g +\int_{\partial M}  \langle \iota^* (F\lrcorner d \omega), \nu \lrcorner d \omega \rangle d\mu_g \nonumber \\
  && + \int_{M}  \langle \delta \omega, \delta(F \lrcorner d \omega) \rangle  d\mu_g - \int_{\partial M}  \langle  \iota^* (\delta \omega), \nu \lrcorner F \lrcorner d \omega \rangle d\mu_g
   \nonumber \\
   &=& \int_{M}  \langle d \omega, \mathcal{L}_F (d \omega) \rangle  d\mu_g +\int_{\partial M}  \langle \iota^* (F\lrcorner d \omega), \nu \lrcorner d \omega \rangle d\mu_g \nonumber \\
  && + \int_{M}  \langle \delta \omega, \delta(F \lrcorner d \omega) \rangle  d\mu_g - \int_{\partial M}  \langle  \iota^* (\delta \omega), \nu \lrcorner F \lrcorner d \omega \rangle d\mu_g
   \nonumber \\
   &=& \int_{M}  \langle d \omega, \nabla_F d \omega \rangle  d\mu_g + \int_{M}  \langle d \omega, T_F^{[p+1]} d \omega \rangle  d\mu_g  +\int_{\partial M}  \langle \iota^* (F\lrcorner d \omega), \nu \lrcorner d \omega \rangle d\mu_g \nonumber \\
  && + \int_{M}  \langle \delta \omega, \delta(F \lrcorner d \omega) \rangle  d\mu_g - \int_{\partial M}  \langle  \iota^* (\delta \omega), \nu \lrcorner F \lrcorner d \omega \rangle d\mu_g.
\end{eqnarray}
We evaluate now the term $\int_{M}  \langle d \omega, \nabla_F d \omega \rangle
d\mu_g$:
\begin{eqnarray*}
\int_{M}  \langle d \omega, \nabla_F d \omega \rangle  d\mu_g &=& \frac{1}{2} \int_{M} F (|d \omega|^2) d\mu_g \\
&=&-\dfrac{1}{2}\int_{M}  |d \omega|^2  \dive\, F  d\mu_g
-\dfrac{1}{2}\int_{\partial M} |d\omega|^2\langle F,\nu\rangle d\mu_g.
\end{eqnarray*}
We recall now that for any vector field $X$ and $(e_i)_{1\leq i\leq n}$ an orthonormal frame we have \begin{align*}
\delta(X \lrcorner \omega)
&= - e_i \lrcorner \nabla_{e_i} (X \lrcorner \omega) \\
&= - e_i \lrcorner (\nabla_{e_i} X \lrcorner \omega + X \lrcorner \nabla_{e_i} \omega) \\
&= - e_i \lrcorner \nabla_{e_i} X \lrcorner \omega + X \lrcorner 
\underbrace{e_i \lrcorner \nabla_{e_i} \omega}_{-\delta \omega}.
\end{align*}
In our case we need $\delta(F\lrcorner d\w),$ which will be equal to $-e_i\lrcorner\nabla_{e_i}F\lrcorner d\w-F\lrcorner\delta d\w$. Now to compute the term $\int_{M}  \langle \delta \omega, \delta(F \lrcorner d \omega) \rangle  d\mu_g$ we use that
$$\delta (F \lrcorner d\omega)= dF \lrcorner d\omega-F \lrcorner \delta d \omega = dF \lrcorner d\omega-F \lrcorner \Delta  \omega + F \lrcorner d \delta \omega$$ where $dF=\sum_{i=1}^ne_i\wedge\nabla_{e_i}F$.
By replacing this last equalities in \eqref{Rellich_calcul}, we obtain
\begin{align}
\label{Rellich_suite}
&\int_{M} \langle \Delta \omega , F\lrcorner d \omega \rangle \, d\mu_g 
 +\int_M \langle \delta \omega , F\lrcorner \Delta \omega \rangle \, d\mu_g \nonumber \\[1ex]
&\quad = -\dfrac{1}{2}\int_{\partial M}  |d \omega|^2  \langle F , \nu \rangle \,  d\mu_g 
 -\dfrac{1}{2}\int_{M}  |d \omega|^2  \operatorname{div}(F) \,  d\mu_g 
 + \int_{M}  \langle d \omega, T_F^{[p+1]} d \omega \rangle \,  d\mu_g \nonumber \\
&\quad\quad +\int_{\partial M}  \langle \iota^* (F\lrcorner d \omega), \nu \lrcorner d \omega \rangle \, d\mu_g 
 +\int_{M}  \langle \delta \omega, dF \lrcorner d\omega \rangle \,  d\mu_g \nonumber \\
&\quad\quad +\int_{M}  \langle \delta \omega, F \lrcorner d \delta \omega \rangle \,  d\mu_g 
 - \int_{\partial M}  \langle  \iota^* (\delta \omega), \nu \lrcorner F \lrcorner d \omega \rangle \, d\mu_g.
\end{align}
It remains to calculate the term $\int_{M}  \langle \delta \omega, F \lrcorner d \delta \omega \rangle  d\mu_g$. We have:
\begin{eqnarray*}
  \int_{M}  \langle \delta \omega, F \lrcorner d \delta \omega \rangle  d\mu_g &=& \int_{M}  \langle \delta \omega, \mathcal{L}_F( \delta \omega) \rangle  d\mu_g -\int_{M}  \langle \delta \omega,d( F \lrcorner  \delta \omega) \rangle  d\mu_g\\
   &=&  \int_{M}  \langle \delta \omega, \nabla_F \delta \omega \rangle  d\mu_g + \int_{M}  \langle \delta \omega, T_F^{[p-1]} \delta \omega \rangle  d\mu_g \\
   &+& \int_{\partial M}  \langle \iota^* (F\lrcorner \delta \omega), \nu \lrcorner \delta \omega \rangle d\mu_g\\
   &=& \int_M\frac{1}{2}F(\lvert \delta\w\rvert^2)d\mu_g+ \int_{M}  \langle \delta \omega, T_F^{[p-1]} \delta \omega \rangle  d\mu_g \\
   &+& \int_{\partial M}  \langle \iota^* (F\lrcorner \delta \omega), \nu \lrcorner \delta \omega \rangle d\mu_g\\
   &=& \frac{1}{2}\int_M\lvert\delta\w\rvert^2\delta F d\mu_g -\frac{1}{2}\int_{\partial M}\lvert\delta\w\rvert^2\langle F,\nu\rangle d\mu_g\\
   &+& \int_{M}  \langle \delta \omega, T_F^{[p-1]} \delta \omega \rangle  d\mu_g + \int_{\partial M}  \langle \iota^* (F\lrcorner \delta \omega), \nu \lrcorner \delta \omega \rangle d\mu_g\\
   &=&-\frac{1}{2}\int_M\lvert\delta\w\rvert^2 \dive (F) d\mu_g-\frac{1}{2}\int_{\partial M}\lvert\delta\w\rvert^2\langle F,\nu\rangle d\mu_g\\
   &+& \int_{M}  \langle \delta \omega, T_F^{[p-1]} \delta \omega \rangle  d\mu_g + \int_{\partial M}  \langle \iota^* (F\lrcorner \delta \omega), \nu \lrcorner \delta \omega \rangle d\mu_g.
\end{eqnarray*}
This last identity combined to \eqref{Rellich_suite} leads to the result. \end{proof}
\section{A new biharmonic Steklov problem with Dirichlet boundary conditions on differential forms} \label{sec:DBS2}
In this section, to define the DBS2, we follow the same methods that are used for the study of DBS \eqref{eq:DBSF1} in \cite{BSFidaGeorgeOlaNicolas}. These methods have also been used for NBS problems in \cite{Article1}. In order to establish the desired inequality, we need a variational characterization of the eigenvalues of the biharmonic Steklov problem with Dirichlet boundary conditions distinct from \eqref{eq:q_1Standard} and \eqref{eq:q_1Alternative}. This characterization is needed to be of the following form:  \begin{align*}
\mathbf{q_{1,p}} = \inf\Bigg\{ &\frac{\lVert \Delta \w \rVert^2_{L^2(M)}}{\lVert \nu \lrcorner d\w \rVert^2_{L^2(\partial M)}} \;\Bigg|\; \w \in \Omega^p(M),\ \w_{|\partial M} = 0,\ \iota^*\delta \w=0 \ \text{on}\ \partial M \ \text{and}\ \nabla_\nu \w \neq 0\Bigg\}.
\end{align*}

To find a spectral problem with such a variational characterization, we compute the critical points of the following functional:
\[
Q(\w):=\frac{\lVert \Delta \w \rVert^2_{L^2(M)}}{\lVert \nu \lrcorner d\w \rVert^2_{L^2(\partial M)}}
\]
that will be the eigenforms of the new problem that we denote by DBS2. Let the space
\[
H^2_{DBS2}(M)=\{\w\in H^2(M) \,|\, \iota^*\delta\w=0 \text{ on } \partial M \text{ and } \w_{|\partial M}=0\}.
\]
The form $\w$ is a critical point if and only if for all $\w'\in H^2_{DBS2}(M)$ we have \[
d_{\w} Q(\w') = \frac{2(\Delta\w,\Delta\w')_{L^2(M)}\lVert\nu\lrcorner d\w\rVert^2_{L^2(\partial M)} - 2(\nu\lrcorner d\w,\nu\lrcorner d\w')_{L^2(\partial M)}\lVert\Delta\w\rVert^2_{L^2(M)}}{\lVert\nu\lrcorner d\w\rVert^4_{L^2(\partial M)}}=0.
\] 

This implies that \[(\Delta\w,\Delta\w')_{L^2(M)}\lVert\nu\lrcorner d\w\rVert^2_{L^2(\partial M)} = (\nu\lrcorner d\w,\nu\lrcorner d\w')_{L^2(\partial M)}\lVert\Delta\w\rVert^2_{L^2(M)}.\] 

Hence, performing an integration by parts using \eqref{eq:IPP2} we get \[(\Delta^2\w,\w')_{L^2(M)} -(\iota^*\Delta\w,\nu\lrcorner d\w')_{L^2(\partial M)}=\frac{\lVert \Delta\w\rVert^2_{L^2(M)}}{\lVert \nu\lrcorner d\w\rVert^2_{L^2(\partial M)}}(\nu\lrcorner d\w,\nu\lrcorner d\w')_{L^2(\partial M)};\]
Therefore,  \begin{equation}\label{eq:delta2=0}
(\Delta^2\w,\w')_{L^2(M)} = \left(\frac{\lVert \Delta\w\rVert^2_{L^2(M)}}{\lVert \nu\lrcorner d\w\rVert^2_{L^2(\partial M)}}\nu\lrcorner d\w-\iota^*\Delta\w,\nu\lrcorner d\w'\right)_{L^2(\partial M)}.
\end{equation}
Since $\w'_{|\partial M} = 0$ and $\iota^*\delta\w'=0$, this implies that $\nu\lrcorner\nabla_\nu\w'=0$ and $\nu\lrcorner d \w'=\iota^*(\nabla_\nu\w').$ Following the argument in \cite{Article1}, we fix a positive $\epsilon$ and define the set $D_\epsilon:=\{x\in M\ | \ d(x,\partial M)\geq \epsilon\}$. We consider a test form $\w'$ vanishing on the set of points whose distance from the boundary is less than $\epsilon$. We have from \eqref{eq:delta2=0} that $(\Delta^2\w,\w')_{L^2(M)}=0$ for all $\w'$, hence $\Delta^2\w=0$ on $D_\epsilon$. Letting $\epsilon$ tend to zero, we obtain that $\Delta^2\w=0$ on $M$, thus the differential form $\nu\lrcorner d\w'$ can take any value, so $\frac{\lVert\Delta\w\rVert^2_{L^2(M)}}{\lVert\nu\lrcorner d\w\rVert^2_{L^2(\partial M)}}\nu\lrcorner d\w=\iota^*\Delta\w.$ We then conclude the following problem:
\begin{equation}(DBS2)
\begin{cases}
    \Delta^2 \w = 0 & \text{on } M \\
     \w = 0 &  \text{on } \partial M\\
    \iota^*  \delta\w = 0 & \text{on } \partial M \\
     \iota^*  \Delta \w- \mathbf{q}\nu \lrcorner d \w   = 0 & \text{on } \partial M.
\end{cases}
\end{equation} 

Let us show that this problem is elliptic. We follow the same steps as in \cite{Article1} to evaluate the principal symbols of $B_1\w=\w_{|\partial M}, \ B_2\w=\iota^*\delta\w,$ and $B_3\w=\iota^*\Delta\w,$ for $v\in T^*M\setminus\{0\}$ and $y(t)=e^{-\lvert v\rvert t}(at+b)\w_0$ with $a, b \in \mathbb{R}$, $\w_0 \in \Lambda^p T^*_x M_{|\partial M}$. We obtain:  
\begin{align*}
\bullet \quad \sigma_{B_1}\!\left(-iv+\partial_t\nu^\flat\right)(y)(0) 
  &= b \w_0\\[6pt]
\bullet \quad  \sigma_{B_2}\!\left(-iv+\partial_t\nu^\flat\right)(y)(0) 
  &= -a\,\nu \lrcorner \w_0 
     + \lvert v\rvert\, b\, \nu \lrcorner \w_0
     + i b\, v \lrcorner \iota^* \w_0 \\[6pt]
\bullet  \quad \sigma_{B_3}\!\left(-iv+\partial_t\nu^\flat\right)(y)(0) 
  &= 2a \lvert v\rvert\, \iota^* \w_0 .
\end{align*}

Thus, we set the map \begin{align*}
\Phi : \mathcal{M}^+_{\Delta^2,v} &\longrightarrow \bigoplus_{j=1}^4 E_j \\
y &\longmapsto \Phi(y)
\end{align*} given by
\[
\Phi(y)=(b\w_0,-a\nu\lrcorner\w_0+\lvert v\rvert b\nu\lrcorner\w_0+ibv\lrcorner\iota^*\w_0,2a\lvert v\rvert \iota^*\w_0).
\]It is easy to check that $\Phi$ is injective and thus is an isomorphism as the dimension of both spaces is $2\binom{n}{p}$. Thus, the problem DBS2 is elliptic in the sense of Shapiro-Lopatinskij.\\
Now to show that the problem DBS2 has a discrete spectrum of eigenvalues, we follow the same steps as in \cite[Section 4.1]{Article1} and \cite[Section 2.2]{BSFidaGeorgeOlaNicolas}. We define the space \[
Z_{DBS2}=\{\w\in\Omega^p(M)\ |\ \Delta^2\w=0, \ \w_{|\partial M}=0, \ \text{and} \ \iota^*\delta\w=0 \text{ on } \partial M\}.
\]   For $\w$ an eigenform and $\w'\in Z_{DBS2}$, we easily obtain through integration by parts \eqref{eq:IPP2} that  
\[
\int_M\langle\Delta\w,\Delta\w'\rangle d\mu_g=\mathbf{q}\int_{\partial M}\langle\nu\lrcorner d\w,\nu\lrcorner d\w'\rangle d\mu_g.
\]  
Referring to the proof of \cite[Section 4]{Article1}, we get the following theorem:
\begin{thm}
Let $(M^n,g)$ be a compact Riemannian manifold with a smooth boundary $\partial M$. The following boundary-value problem:  
\begin{equation}\label{eq:DBSF2}
    (DBS2)
    \left\{
    \hspace{-0.3em}
    \begin{array}{r@{\quad}c@{\quad}l@{\quad}l}
        \Delta^2 \w &=& 0 & \text{on } M \\[0.15em]
        \w &=& 0 & \text{on } \partial M \\[0.15em]
        \iota^* \delta\w &=& 0 & \text{on } \partial M \\[0.15em]
        \iota^* \Delta \w - \mathbf{q}\nu \lrcorner d\w &=& 0 & \text{on } \partial M
    \end{array}
    \right.
\end{equation} 
on $p$-forms, has a discrete spectrum consisting of a non decreasing sequence of positive eigenvalues of finite multiplicities, denoted by $(\mathbf{q_{i,p}})_{i\geq 1}.$

\end{thm}
Thus, we have the following variational characterizations, whose proofs are identical to those of the variational characterizations \eqref{eq:q_1Standard} treated in \cite{Article1} and \cite[Theorem 2.6]{BSFidaGeorgeOlaNicolas}:

    \begin{thm}A variational characterization of the first eigenvalue of DBS2 is given by
\begin{equation} \label{eq:1vpDBS2}
\begin{aligned}
\mathbf{q_{1,p}} = \inf \biggl\{ 
&\frac{\lVert \Delta \w \rVert^2_{L^2(M)}}{\lVert \nu\lrcorner d\w\rVert^2_{L^2(\partial M)}} 
\ \biggm|\ 
\w\in \Omega^p(M),\ \w_{|\partial M}=0,\ \iota^*\delta\w=0\ \text{on}\ \partial M, \ \text{and}\ \nabla_\nu\w\neq 0 \biggr\}.
\end{aligned}
\end{equation}
\end{thm}
\begin{prop}
      For all $k\in\mathbb{N}$, the variational characterization of the $k^{\text{th}}$ eigenvalue of the DBS2 on differential forms is given by:
\begin{equation}\label{eq:kvpDBS2}
    \mathbf{q_{k,p}}=\min_{\substack{\mathcal{V}_k\subset H^2_{DBS2}(M) \\ \dim (\mathcal{V}_k/H^2_{DBS2}(M)\cap \mathcal{V}_k)=k}}\ \underset{\w\in \mathcal{V}_k\backslash H^2_{DBS2}(M)}{\max}\frac{\lVert \Delta \w \rVert^2_{L^2(M)}}{\lVert \nu\lrcorner d\w\rVert^2_{L^2(\partial M)}}.\end{equation}\end{prop}
\begin{lemma}\label{lem:q<q}
We use the notations defined previously, where $q_{k,p}$ is the characterization of the $k^{\text{th}}$ eigenvalue of \eqref{eq:DBSF1} and $\mathbf{q_{k,p}}$ is the characterization of the $k^{\text{th}}$ eigenvalue of \eqref{eq:DBSF2}. We have the following inequality:
    \[
    q_{k,p} \leq \mathbf{q_{k,p}}.
    \]
\end{lemma}

\begin{proof}
    First of all the quotient of $\mathbf{q_{k,p}}$ is bigger than the one of $q_{k,p}$ as we omit from the denominator the term $\lVert\iota^*\delta\w\rVert^2_{L^2(\partial M)}$. In addition , the space where $\mathbf{q_{k,p}}$ is defined is smaller than the one of $q_{k,p}$, and we know that the infimum over the larger set is the smaller. Both facts explain the order given for the variational characterizations.
\end{proof}

\section{Eigenvalue estimates}\label{sec:eigenvalueEstimates}
In this section, we establish Kuttler-Sigillito inequalities with geometric conditions on the curvature of the manifold. We start by the following lemma that will be used later on.
\begin{lemma}\label{lm:nablaXW<XnablaW}
    Let $\w \in \Omega^p(M)$ and $X \in \chi(M)$, then the following inequality holds: 
    \[
    \lvert \nabla_X\w \rvert \leq \lvert X \rvert \cdot \lvert \nabla\w \rvert.
    \]
\end{lemma}

\begin{proof} Let $\{e_j\}_{1\leq j\leq n}$ be an orthonormal frame of $TM.$ We compute
   \begin{align*}
\lvert\nabla_X\w\rvert^2 &= \sum_{j,k=1}^n \langle X,e_j\rangle \cdot \langle X,e_k\rangle \cdot \langle\nabla_{e_j}\w,\nabla_{e_k}\w\rangle \\
&\leq \sum_{j,k=1}^n \lvert \langle X,e_j\rangle \rvert \cdot \lvert \langle X,e_k\rangle \rvert \cdot \lvert \nabla_{e_j}\w \rvert \cdot \lvert \nabla_{e_k}\w \rvert \\
&= \left(\sum_{j=1}^n \lvert \langle X,e_j \rangle \rvert \cdot \lvert \nabla_{e_j}\w \rvert \right)^2 \\
&\leq \left(\sum_{j=1}^n \langle X,e_j \rangle^2 \right) \cdot \left(\sum_{j=1}^n \lvert \nabla_{e_j}\w \rvert^2 \right) \\
&= \lvert X \rvert^2 \cdot \lvert \nabla\w \rvert^2.
\end{align*}
\end{proof}
For a reminder of $W^{[p]}$ and $S^{[p]}$ see \thref{thm:bochner} and \defref{def:shape}. Recall that $\sigma_{1,p}$ is the first non-zero eigenvalue of the Steklov operator. The following theorem gives us Inequality \eqref{eq:sigma1>mu1/mu_1}.
\begin{thm}\label{thm:sigma1>mu1...PourLaPreuve}
     Let $(M^n,g)$ be a compact Riemannian manifold with smooth boundary $\partial M$. Let $1\leq p\leq n$, assume that $M$ is star-shaped with respect to $x_0$ and that $\ric_g \geq (n-1)\kappa$ where $\kappa \in \mathbb{R}$, $W^{[p]} \geq 0$ and $S^{[p]} \geq 0$. Then, we have the following inequality: 
\begin{equation}\label{eq:sigma1>mu1/mu_1ForTheProof}
      \sigma_{1,p} > \frac{\frac{1}{2}h_{\min} \mu_{1,p}}{r_{\max} \mu_{1,p}^{\frac{1}{2}} + \frac{1}{2} \underset{M}{\max}(1 + d_{x_0}.s_\kappa \circ d_{x_0})}.
\end{equation}
\end{thm}

\begin{remark}
 We have particular cases, where the hypothesis of \thref{thm:sigma1>mu1...PourLaPreuve} can be reformulated as follows:
 \begin{enumerate}[label=\roman*.]
     \item[$\bullet$] If $p=1$ and $\kappa>0$, then the assumptions of the Theorem reduce to the following:
$W^{[1]}=\ric\geq (n-1)\kappa$ and $S^{[1]}=S \geq 0$.
     \item[$\bullet$]  If $p = 1$ and $\kappa=0$, then the assumptions of the Theroem reduce to $W^{[1]}=\ric\geq0$ and $S^{[1]}=S\geq0$, the inequality \eqref{eq:sigma1>mu1/mu_1ForTheProof} becoming
\begin{equation*}
      \sigma_{1,p} > \frac{\frac{1}{2}h_{\min} \mu_{1,p}}{r_{\max} \mu_{1,p}^{\frac{1}{2}} + \frac{1}{2} \underset{M}{\max}(1 + d_{x_0}.s_0 \circ d_{x_0})}.
\end{equation*}
 \end{enumerate}
\end{remark}

\begin{proof}[Proof of \thref{thm:sigma1>mu1...PourLaPreuve}]
  Let $u \in \Omega^p(M)$ be a $\sigma_{1,p}$-eigenform such that $\iota^*u$ is ${L^2(\partial M)}$-orthogonal to $H_{\Abs}(M)$ and let $\w := \proj_{\perp_{L^2(M)}}(u)$ in $H_{\Abs}(M)^\perp$. We use the following identity which follows from integration by parts:
     \[\int_M\frac{1}{2}\lvert\w\rvert^2\Delta\rho_{x_0} d\mu_g-\int_M\langle\nabla\rho_{x_0},\nabla(\frac{1}{2}\lvert\w\rvert^2)\rangle d\mu_g=\frac{1}{2}\int_{\partial M}\lvert\w\rvert^2\langle\nu,\nabla\rho_{x_0}\rangle d\mu_g.\] Since $X(\frac{1}{2}\lvert \w\rvert^2)=\langle\nabla_X\w,\w\rangle$, replacing $X$ by $\nabla\rho_{x_0}$ we get \[\int_M\frac{1}{2}\lvert\w\rvert^2\Delta\rho_{x_0} d\mu_g-\int_M\langle\nabla_{\nabla\rho_{x_0}}\w,\w\rangle d\mu_g=\frac{1}{2}\int_{\partial M}\lvert\w\rvert^2\cdot\partial_\nu\rho_{x_0} d\mu_g.\] Using \thref{thm:laplace}, we get \begin{align}
-\frac{1}{2} \int_{\partial M} \lvert \w \rvert^2 \cdot \partial_\nu \rho_{x_0} \, d\mu_g \nonumber
&\leq \int_M \langle \nabla_{\nabla \rho_{x_0}} \w, \w \rangle \, d\mu_g + \frac{1}{2} \int_M \lvert \w \rvert^2 (1 + d_{x_0}.s_\kappa \circ d_{x_0}) \, d\mu_g \nonumber\\
&\leq \int_M \langle \nabla_{\nabla \rho_{x_0}} \w, \w \rangle \, d\mu_g + \frac{1}{2} \underset{M}{\max}(1 + d_{x_0}.s_\kappa \circ d_{x_0}) \int_M \lvert \w \rvert^2 \, d\mu_g. \label{eq:preuveineg1}
\end{align}
Now, according to \lemref{lm:nablaXW<XnablaW}, we obtain $\lvert \langle \nabla_{\nabla \rho_{x_0}} \w, \w \rangle \rvert \leq \lvert \nabla \rho_{x_0} \rvert \cdot \lvert \nabla \w \rvert \cdot \lvert \w \rvert$. Notice that $\lvert \nabla\rho_{x_0}\rvert=d_{x_0}$. We then compute using Hölder inequality:
\begin{align}
\left( \int_M \langle \nabla_{\nabla \rho_{x_0}} \w, \w \rangle \, d\mu_g \right)^2 
&\leq \left(\int_M d_{x_0}\lvert\nabla\w\rvert \cdot\lvert \w \rvert \ d\mu_g\right)^2 \nonumber\\ \nonumber
&\leq \int_M \lvert \w \rvert^2 \, d\mu_g \int_M d_{x_0}^2 \lvert \nabla \w \rvert^2 \, d\mu_g \\ \nonumber
&\leq \left(\underset{M}{\max} \ d_{x_0}^2\right) \left(\int_M \lvert \w \rvert^2 \, d\mu_g\right)\left( \int_M \lvert \nabla \w \rvert^2 \, d\mu_g\right) \\ 
&= r_{\max}^2 \left(\int_M \lvert \w \rvert^2 \, d\mu_g\right)\left( \int_M \lvert \nabla \w \rvert^2 \, d\mu_g\right).\label{eq:preuveineg2}\end{align} 

According to Reilly's formula \cite[Theorem 3]{RaulotSavo2}, and for $\nu \lrcorner \w = 0$, we have:
 \[\left(\int_M \lvert d\w\rvert^2+\lvert\delta\w\rvert^2\right)d\mu_g=\int_M\lvert\nabla\w\rvert^2d\mu_g+ \int_M\langle W^{[p]}\w,\w\rangle d\mu_g+\int_{\partial M} \langle S^{[p]}(\iota^*\w),\iota^*\w\rangle d\mu_g.\]
Hence, for $S^{[p]}\geq 0$ and $W^{[p]}\geq 0$ we find\begin{equation}\label{eq:preuveineg3}
\int_M\lvert\nabla\w\rvert^2 d\mu_g\leq \int_M \left(\lvert d\w\rvert^2+\lvert\delta\w\rvert^2\right)d\mu_g.\end{equation} 
Plugging Equations \eqref{eq:preuveineg2} and \eqref{eq:preuveineg3} into \eqref{eq:preuveineg1}, we get
\begin{align*}
\frac{1}{2}\int_{\partial M}h_{\min}\lvert\w\rvert^2 d\mu_g&\leq-\frac{1}{2}\int_{\partial M} \lvert \w \rvert^2 \partial_\nu \rho_{x_0} \, d\mu_g 
\leq r_{\max} \left( \int_M \lvert \w \rvert^2 \, d\mu_g \right)^{\frac{1}{2}} \left( \int_M \left(\lvert d\w \rvert^2 + \lvert \delta \w \rvert^2\right) \, d\mu_g \right)^{\frac{1}{2}} \\
&\quad + \frac{1}{2} \underset{M}{\max}(1 + d_{x_0}.s_\kappa \circ d_{x_0}) \int_M \lvert \w \rvert^2 \, d\mu_g \\
&\overset{\eqref{eq:vpNeumannAbsolueForme}}{\leq} r_{\max} \mu_{1,p}^{-\frac{1}{2}} \int_M \lvert \left(d\w \rvert^2 + \lvert \delta \w \rvert^2\right) \, d\mu_g \\
&\quad + \frac{1}{2} \underset{M}{\max}(1 + d_{x_0}.s_\kappa \circ d_{x_0}) \mu_{1,p}^{-1} \int_M \left(\lvert d\w \rvert^2 + \lvert \delta \w \rvert^2\right) \, d\mu_g\\
&\leq \left(r_{\max} \mu_{1,p}^{-\frac{1}{2}}+ \frac{1}{2} \underset{M}{\max}(1 + d_{x_0}.s_\kappa \circ d_{x_0}) \mu_{1,p}^{-1} \right)\int_M \left(\lvert d\w \rvert^2 + \lvert \delta \w \rvert^2\right) \, d\mu_g.
\end{align*}
Recall that $u=\w+u_0$ with $u_0\in H_{\Abs}(M)$ and $\w\perp_{L^2(M)}u_0,$ then \begin{equation}\label{eq:u=w+u0}
\lVert\w\rVert^2_{L^2(\partial M)}=\lVert u-u_0\rVert^2_{L^2(\partial M)}=\lVert u\rVert^2_{L^2(\partial M)}+\lVert u_0\rVert^2_{L^2(\partial M)}-2\underbrace{(u,u_0)_{L^2(\partial M)}}_{0}\geq \lVert u\rVert^2_{L^2(\partial M).}\end{equation} Hence we find

\begin{align}
\frac{1}{2} h_{\min} \int_{\partial M} \lvert u \rvert^2 \, d\mu_g \nonumber
&\leq \frac{1}{2} h_{\min} \int_{\partial M} \lvert \w \rvert^2 \, d\mu_g \\ \nonumber
&\leq \left( r_{\max} \mu_{1,p}^{-\frac{1}{2}} + \frac{1}{2} \underset{M}{\max}(1 + d_{x_0}.s_\kappa \circ d_{x_0}) \mu_{1,p}^{-1} \right) \int_M \left(\lvert d\w \rvert^2 + \lvert \delta \w \rvert^2\right) \, d\mu_g \\\nonumber
&= \left( r_{\max} \mu_{1,p}^{-\frac{1}{2}} + \frac{1}{2} \underset{M}{\max}(1 + d_{x_0}.s_\kappa \circ d_{x_0}) \mu_{1,p}^{-1} \right) \int_M \left(\lvert d u \rvert^2 + \lvert \delta u \rvert^2 \right)\, d\mu_g\\ \nonumber
&\overset{\eqref{eq:sigma1formesSteklov}}{=}\sigma_{1,p} \left( r_{\max} \mu_{1,p}^{-\frac{1}{2}} + \frac{1}{2} \underset{M}{\max}(1 + d_{x_0}.s_\kappa \circ d_{x_0}) \mu_{1,p}^{-1} \right)\int_{\partial M}\lvert u \rvert^2 d\mu_g.\nonumber
\end{align}
Thus, we deduce the following inequality:
\begin{equation*}
      \sigma_{1,p}\geq \frac{\frac{1}{2}h_{\min} \mu_{1,p}}{r_{\max}\mu_{1,p}^{\frac{1}{2}}+\frac{1}{2}\underset{M}{\max}(1+d_{x_0}.s_\kappa\circ d_{x_0})}.\end{equation*}

Next, we show that this inequality is strict. Suppose that
\begin{equation*}
\sigma_{1,p} = \frac{\frac{1}{2} h_{\min} \, \mu_{1,p}}{r_{\max} \mu_{1,p}^{\frac{1}{2}} + \frac{1}{2} \underset{M}{\max}\bigl(1 + d_{x_0} \cdot s_\kappa \circ d_{x_0}\bigr)}.
\end{equation*}
Then we have equalities in all the above inequalities, which leads to \(\omega = u\), 
resulting in a contradiction, because \(u\) would simultaneously be a Steklov and a Neumann eigenform and therefore must be zero. Hence, the inequality must be strict.

\end{proof}

\begin{cor}
    If we let $M$ to be a bounded domain in $\mathbb{R}^n$ with convex smooth boundary $\partial M$ (convexity gives that $S\geq0$). Then, for $p\geq 1,$ we have the following inequality: 
    \[\sigma_{1,p}>\frac{\frac{1}{2}h_{\min}\mu_{1,p}}{r_{\max}\mu_{1,p}^{\frac{1}{2}}+\frac{n}{2}}.\]
    \end{cor}
    This holds since $\ric=0$ and $s_\kappa(r)=\frac{n-1}{r}$ on $\mathbb{R}^n$. In particular, on the ball of radius $r$, the inequality becomes
  \[\sigma_{1,p}>\frac{{r}\mu_{1,p}}{2r\mu_{1,p}^{\frac{1}{2}}+{n}}.\] 

Let us show Inequality \eqref{eq:th1.4(i,2)} of \thref{thm:ineg2}.
\begin{thm}\label{thm:inegalité2}
Let $(M^n,g)$ be a compact Riemannian manifold. Assume that $M$ is star-shaped with respect to $x_0\in M$ and that its sectional curvature $\kappa_g$ satisfies $\kappa_1 \leq \kappa_g \leq \kappa_2$. Then, for all $k\in\mathbb{N}$, $1\leq p\leq n-1$, there exists a constant $c_1 = c_1(p,n,\kappa_1,\kappa_2)$ such that:
\begin{equation}\label{eq:th1.4(i,2)PourLaPreuve}
     \lambda_{k,p}\leq \frac{4\mathbf{q}^2_{k,p}r^2_{\max}+2\mathbf{q} _{k,p}h_{\min} c_1}{h^2_{\min}}.
\end{equation}
\end{thm}
Let $(e_{j})_{j}$ be a local eigenvector basis of $\nabla^2 \rho_{x_0}$, such that $(\nabla^2 \rho_{x_0})\, (e_{j}) = \gamma_{j}\, e_{j}.$
For $p \geq 1$, in order to bound the term $\langle T^{[p]}\cdot,\cdot\rangle$ appearing in the Rellich identity \eqref{Rellich_forms}, we define
\[
\beta_{\min,p}(y)
:= \min_{j_1 < \cdots < j_p} \sum_{k=1}^p \gamma_{j_k}(y)
\qquad \text{and} \qquad
\beta_{\max,p}(y)
:= \max_{j_1 < \cdots < j_p} \sum_{k=1}^p \gamma_{j_k}(y).
\]

\begin{lemma}
    Let $(M^n,g)$ be a compact Riemannian manifold. Assume that $M$ is star-shaped with respect to $x_0\in M$ then we have \begin{equation}\label{eq:encadrementBetta}
    \beta_{\min,p}(y)\lvert\w\rvert^2\leq \langle T_F^{[p]}\w,\w\rangle(y)\leq \beta_{\max,p}(y)\lvert\w\rvert^2.\end{equation} In particular, if the sectional curvature $\kappa_g$ satisfies $\kappa_1 \leq \kappa_g \leq \kappa_2$. Then we get, at each point of $M$: \begin{equation}\label{eq:beta}
  \frac{p}{n-1}\underset{M}{\min}\ (d_{x_0}.s_{\kappa_2}(d_{x_0}))\cdot\lvert\w\rvert^2\leq \langle T_F^{[p]}\w,\w\rangle\leq \frac{p}{n-1}\underset{M}{\max}\ (d_{x_0}.s_{\kappa_1}(d_{x_0}))\cdot\lvert\w\rvert^2.
\end{equation}
\end{lemma}
\begin{proof}
We compute the following for $F=\nabla\rho_{x_0}$:
\begin{align*}
    (T_F^{[p]}\w)(e_{j_1},\ldots,e_{j_p})&=\sum_{k=1}^p\w(e_{j_1},\ldots,\nabla_{e_{j_k}}F,\ldots,e_{j_p})\\
    &=\sum_{k=1}^p\w(e_{j_1},\ldots,\nabla^2\rho_{x_0}({e_{j_k}}),\ldots,e_{j_p})\\
    &=\sum_{k=1}^p\gamma_{j_k}\w(e_{j_1},\ldots,e_{j_p}).
\end{align*}
Thus, we can write the following bounds using \thref{thm:hess}: 
\begin{equation*}
    \underbrace{\left(\underset{{j_1<\ldots<j_p}}{\min}\sum_{k=1}^p\gamma_{j_k}(y)\right)}_{\beta_{\min,p}(y)}\lvert\w\rvert^2\leq \langle T_F^{[p]}\w,\w\rangle(y)\leq \underbrace{\left(\underset{{j_1<\ldots<j_p}}{\max}\sum_{k=1}^p\gamma_{j_k}(y)\right)}_{\beta_{\max,p}(y)}\lvert\w\rvert^2.\end{equation*}
\end{proof}

\begin{proof}[Proof of \thref{thm:inegalité2}]
Let $\w_1,\ldots,\w_k \in H^2(M)$ be $\mathbf{q_{1,p}},\ldots,\mathbf{q_{k,p}}$-eigenforms and $E_k = \spann(\w_1,\ldots,\w_k)$. Thus for any $\w\in E_k$ \begin{align}\label{eq:lamda<qsup}
\lambda_{k,p}  &\overset{\eqref{eq:vpKAlternativeDirichletForme}}{\leq} \underset{\w \in E_k}{\sup}  
\frac{\lVert \Delta \w \rVert^2_{L^2(M)}}{\lVert d\w \rVert^2_{L^2(M)} + \lVert \delta\w \rVert^2_{L^2(M)}} \nonumber \\
&\overset{\eqref{eq:kvpDBS2}}{\leq} \mathbf{q_{k,p}} \underset{\w \in E_k}{\sup} 
\frac{\lVert \nu \lrcorner d\w \rVert^2_{L^2(\partial M)}}{\lVert d\w \rVert^2_{L^2(M)} + \lVert \delta\w \rVert^2_{L^2(M)}} \notag \\
&= \mathbf{q_{k,p}}  \underset{\w \in E_k}{\sup}  
\frac{\lVert \nabla_\nu \w \rVert^2_{L^2(\partial M)}}{\lVert d\w \rVert^2_{L^2(M)} + \lVert \delta\w \rVert^2_{L^2(M)}}.
\end{align}

We take $F=\nabla\rho_{x_0},$ the Rellich identity \eqref{Rellich_forms} gives us \begin{align}\label{eq:rellichdanslapreuveDeLinegalite} 
\int_{M}\langle \Delta \omega , F\lrcorner d \omega \rangle d\mu_g &+ \int_M \langle \delta \omega , F\lrcorner \Delta \omega \rangle d\mu_g \nonumber \\
&= -\dfrac{1}{2}\int_{\partial M} (|d \omega|^2 +|\delta \omega|^2) \langle F,\nu \rangle d\mu_g \nonumber \\
&\quad + \int_{\partial M} \langle F \wedge \iota^* (\delta \omega), \nu \lrcorner d \omega \rangle d\mu_g + \int_{\partial M} \langle \iota^* (F\lrcorner d \omega), \nu \lrcorner d \omega \rangle d\mu_g \nonumber \\
&\quad + \int_{\partial M} \langle \iota^* (F\lrcorner \delta \omega), \nu \lrcorner \delta \omega \rangle d\mu_g \nonumber \\
&\quad -\dfrac{1}{2}\int_{M} (|d \omega|^2 +|\delta \omega|^2) \dive(F) d\mu_g + \int_{M} \langle \delta \omega, dF \lrcorner d \omega \rangle d\mu_g \nonumber \\
&\quad + \int_{M} \langle T_F^{[p+1]} d \omega, d \omega \rangle d\mu_g + \int_{M} \langle T_F^{[p-1]} \delta \omega, \delta \omega \rangle d\mu_g.
\end{align} 

We first control the boundary terms: 
\begin{align*}
    &\underbrace{-\frac{1}{2}\int_{\partial M} (|d \omega|^2 + |\delta \omega|^2) \langle F,\nu  \rangle d\mu_g}_{A} 
    + \underbrace{\int_{\partial M} \langle F \wedge \iota^* (\delta \omega), \nu \lrcorner d \omega \rangle d\mu_g}_{B} \\
    &\quad + \underbrace{\int_{\partial M} \langle \iota^* (F \lrcorner d \omega), \nu \lrcorner d \omega \rangle d\mu_g}_{C} 
    + \underbrace{\int_{\partial M} \langle \iota^* (F \lrcorner \delta \omega), \nu \lrcorner \delta \omega \rangle d\mu_g}_{D}.
\end{align*}

\begin{enumerate}
        \item [\underline{A:}] Since $\w\in E_k$, we have $\nu\lrcorner d\w=\iota^*\nabla_\nu\w$ and $\iota^*\delta\w=-\nu\lrcorner\nabla_\nu\w=0.$\\ Thus, $\lVert\nu\lrcorner d\w\rVert^2_{L^2(\partial M)}=\lVert\nabla_\nu\w\rVert^2_{L^2(\partial M)}.$ We also have $$\lvert d\w\rvert^2=\lvert\nu\lrcorner d\w\rvert^2=\lvert\iota^*(\nabla_\nu\w)\rvert^2.$$ Now by \cite[p.901]{RaulotSavo}, we have $\nu\lrcorner\delta\w=-\delta(\nu\lrcorner\w)$, hence, \[\lvert\delta\w\rvert^2=\lvert\nu\lrcorner\delta\w\rvert^2+\lvert\iota^*\delta\w\rvert^2=\lvert-\delta(\underbrace{\nu\lrcorner\w}_{0})\rvert^2+\underbrace{\lvert\iota^*\delta\w\rvert^2}_{0}=\lvert\nu\lrcorner\nabla_\nu\w\rvert^2=0.\]Thus, we finally get that the integral (A) is equal to
 \begin{align}
    -\frac{1}{2}\int_{\partial M} (|d \omega|^2 + |\delta \omega|^2) \langle F,\nu  \rangle d\mu_g \nonumber
    &= -\frac{1}{2}\int_{\partial M} (|\nu\lrcorner d\w|^2) \langle F,\nu  \rangle d\mu_g \label{eq:A}\\
    &= -\frac{1}{2}\int_{\partial M} |\nabla_\nu \omega|^2 \langle F,\nu  \rangle d\mu_g.
\end{align}

\item [\underline{B:}] Since $\iota^*\delta\w=0$, then the integral (B) vanishes: $$\int_{\partial M} \langle F \wedge \iota^* (\delta \omega), \nu \lrcorner d \omega \rangle d\mu_g=\int_{\partial M}\langle\underbrace{\iota^*\delta\w}_{0},F\lrcorner\nu\lrcorner d\w\rangle d\mu_g=0.$$

 \item [\underline{C:}] Since $\w\in E_k$, we know that $\iota^* d\w=d\iota^*\w=0.$ We compute $$\iota^*(F\lrcorner d\w)=\langle F,\nu\rangle \iota^*(\nu\lrcorner d\w)+\iota^*F\lrcorner \iota^* d\w=\langle F,\nu \rangle\iota^*(\nu\lrcorner d\w).$$ Therefore, we have \begin{align*}
    \int_{\partial M} \langle \iota^* (F \lrcorner d \omega), \nu \lrcorner d \omega \rangle d\mu_g 
    &= \int_{\partial M} \langle F , \nu  \rangle \lvert\nu \lrcorner d\w\rvert^2  d\mu_g \\
    &= \int_{\partial M} \langle F , \nu  \rangle \lvert\iota^* (\nabla_\nu \w)\rvert^2  d\mu_g \\
    &= \int_{\partial M} \langle F , \nu  \rangle \big(\lvert\nabla_\nu \w\rvert^2 
       - \underbrace{\lvert\nu \lrcorner \nabla_\nu \w\rvert^2}_{0}\big) d\mu_g\\
       &= \int_{\partial M} \langle F , \nu  \rangle \lvert\nabla_\nu \w\rvert^2 
        d\mu_g.
\end{align*} 
 \item [\underline{D:}]  Since $\nu\lrcorner\delta\w=-\delta(\nu\lrcorner\w)=0$, the integral (D) vanishes:  \begin{equation}\label{eq:B}
     \int_{\partial M} \langle \iota^* (F \lrcorner \delta \omega), \nu \lrcorner \delta \omega \rangle d\mu_g=0.
 \end{equation}
\end{enumerate}

Combining the results and the fact that $dF=0$, \eqref{eq:rellichdanslapreuveDeLinegalite} reduces to 
 \begin{align*}
&-\underbrace{\int_{M} \langle \Delta \omega, F \lrcorner d \omega \rangle d\mu_g 
- \int_M \langle \delta \omega, F \lrcorner \Delta \omega \rangle d\mu_g}_{E} 
 + \underbrace{\frac{1}{2} \int_{M} (|d \omega|^2 + |\delta \omega|^2) \Delta\rho_{x_0} \, d\mu_g }_{F}\\
 &+ \underbrace{\int_{M} \langle T_F^{[p+1]} d \omega, d \omega \rangle \, d\mu_g }_{G} + \underbrace{\int_{M} \langle T_F^{[p-1]} \delta \omega, \delta \omega \rangle \, d\mu_g}_{G}-\int_M\langle\delta\w,\underbrace{dF}_{0}\lrcorner d\w\rangle d\mu_g\\ 
 &= - \frac{1}{2}\int_{\partial M} \lvert\nabla_\nu\w\rvert^2\langle F,\nu\rangle d\mu_g.
\end{align*}
Since  \begin{align*}
    \lvert F\lrcorner\alpha\rvert^2+\lvert F\wedge\alpha\rvert^2&=\lvert F\rvert^2\cdot \lvert \alpha\rvert^2
 \end{align*} we have \begin{equation}\label{eq:fw<pfw}\lvert F\lrcorner\alpha\rvert\leq \lvert F\rvert \cdot \lvert \alpha\rvert \ \text{and}\ \lvert F\wedge\alpha\rvert\leq \lvert F\rvert \cdot \lvert \alpha\rvert\end{equation} that are both valid for any form $\alpha.$ We estimate the remaining integrals separately, using inequalities \eqref{eq:fw<pfw}, Hölder inequality and taking $F=\nabla\rho_{x_0}$.
We begin with 
 \begin{enumerate}
        \item [\underline{E:}] \begin{align*}
&-\int_{M} \langle \Delta \omega, F \lrcorner d \omega \rangle d\mu_g 
 - \int_M \langle \delta \omega, F \lrcorner \Delta \omega \rangle d\mu_g \\ 
&= - \int_M \langle \Delta \omega, F \lrcorner d \omega + F \wedge \delta \omega \rangle d\mu_g \notag \\
&\leq \int_M \lvert \Delta \omega \rvert \cdot \lvert F \lrcorner d \omega + F \wedge \delta \omega \rvert d\mu_g \notag \\
&\leq \left( \int_M \lvert \Delta \omega \rvert^2 d\mu_g \right)^{\frac{1}{2}} 
\left( \int_M \left(\lvert F \lrcorner d \omega + F \wedge \delta \omega \rvert^2\right) d\mu_g \right)^{\frac{1}{2}} \notag \\
&\leq  \left( \int_M \lvert \Delta \omega \rvert^2 d\mu_g \right)^{\frac{1}{2}} 
\left( \int_M \bigl(\lvert F \lrcorner d \omega \rvert^2 + \lvert F \wedge \delta \omega \rvert^2+ 2\underbrace{\langle F\lrcorner d\omega,F\wedge\delta\omega\rangle}_{0}\bigr) d\mu_g \right)^{\frac{1}{2}} \notag \\
&\overset{\eqref{eq:fw<pfw}}{\leq} \left( \int_M \lvert \Delta \omega \rvert^2 d\mu_g \right)^{\frac{1}{2}} 
\left( \int_M (\lvert F \rvert^2 \lvert d \omega \rvert^2 + \lvert F \rvert^2 \lvert \delta \omega \rvert^2)d\mu_g \right)^{\frac{1}{2}} \notag \\
&\leq r_{\max} \left( \int_M \lvert \Delta \omega \rvert^2 d\mu_g \right)^{\frac{1}{2}} 
\left( \int_M (\lvert d \omega \rvert^2 + \lvert \delta \omega \rvert^2) d\mu_g \right)^{\frac{1}{2}} \notag \\
&\overset{\eqref{eq:kvpDBS2}}{\leq}\mathbf{q_{k,p}} ^{\frac{1}{2}} r_{\max}
\left( \int_{\partial M} \lvert \nabla_\nu \omega \rvert^2 d\mu_g \right)^{\frac{1}{2}} \,  \left( \int_M (\lvert d \omega \rvert^2 + \lvert \delta \omega \rvert^2) d\mu_g \right)^{\frac{1}{2}}.
\end{align*}

\item [\underline{F:}] Using again \thref{thm:laplace}, we find\begin{align*}
&\frac{1}{2} \int_{ M} (|d \omega|^2 + |\delta \omega|^2) \Delta\rho_{x_0} \, d\mu_g \notag \leq -\frac{1}{2} \int_{ M} (|d \omega|^2 + |\delta \omega|^2) (1 + d_{x_0}.s_{\kappa_2}(d_{x_0})) \, d\mu_g . \notag 
\end{align*}

\item [\underline{G:}] Using \eqref{eq:encadrementBetta}, we have
\begin{equation*}
    \int_M\langle T_F^{[p+1]}d\w,d\w\rangle d\mu_g\leq \beta_{\max,p+1}\int_M\lvert d\w\rvert ^2 d\mu_g
\end{equation*}
and \begin{equation*}
    \int_M\langle T_F^{[p-1]}\delta\w,\delta\w\rangle d\mu_g\leq \beta_{\max,p-1}\int_M\lvert \delta\w\rvert ^2 d\mu_g.
\end{equation*} Thus, we get
\begin{align*}
  \int_M \langle T_F^{[p+1]} d\omega, d\omega \rangle d\mu_g +  \int_M \langle T_F^{[p-1]} \delta\omega, \delta\omega \rangle d\mu_g & \leq \underset{M}{\max}(\beta_{\max,p+1};\beta_{\max,p-1}) \\
  & \quad \times \int_{M} (|d\omega|^2 + |\delta \omega|^2) d\mu_g.
\end{align*}

        \end{enumerate}

    Therefore, by regrouping the terms we finally obtain: \begin{align*}
            -\frac{1}{2}\int_{\partial M}\lvert\nabla_\nu\w\rvert^2\langle\nabla\rho_{x_0},\nu\rangle d\mu_g &\leq \mathbf{q_{k,p}} ^\frac{1}{2}r_{\max}\left( \int_{\partial M} \lvert \nabla_\nu \omega \rvert^2 d\mu_g \right)^{\frac{1}{2}}\left( \int_M (\lvert d \omega \rvert^2+ \lvert \delta \omega \rvert^2) d\mu_g \right)^{\frac{1}{2}} \\
        &\quad+\left(\underset{M}{\max}(\beta_{\max,p+1};\beta_{\max,p-1})-\frac{1}{2}(1+\underset{M}{\min} (d_{x_0}.s_{\kappa_2}(d_{x_0})))\right)\\
        &\quad\times\int_{ M} (|d \omega|^2 + |\delta \omega|^2)d\mu_g\\
        &\overset{\eqref{eq:beta}}{\leq} \mathbf{q_{k,p}}^\frac{1}{2}r_{\max}\left( \int_{\partial M} \lvert \nabla_\nu \omega \rvert^2 d\mu_g \right)^{\frac{1}{2}}\left( \int_M (\lvert d \omega \rvert^2+ \lvert \delta \omega \rvert^2) d\mu_g \right)^{\frac{1}{2}} \\
        &\quad+\left(\frac{p+1}{n-1}\underset{M}{\max}\ ({d_{x_0}.s_{\kappa_1}(d_{x_0})})-\frac{1}{2}(1+\underset{M}{\min}(d_{x_0}.s_{\kappa_2}(d_{x_0})))\right)\\
        &\quad\times \int_{ M} (|d \omega|^2 + |\delta \omega|^2)d\mu_g. \end{align*}

This implies that \small\begin{align*} 
h_{\min}\int_{\partial M}\lvert\nabla_\nu\w\rvert^2 d\mu_g &\leq 2\mathbf{q_{k,p}}^\frac{1}{2}r_{\max}\left( \int_{\partial M} \lvert \nabla_\nu \omega \rvert^2 d\mu_g \right)^{\frac{1}{2}}\left( \int_M (\lvert d \omega \rvert^2 + \lvert \delta \omega \rvert^2) d\mu_g \right)^{\frac{1}{2}}\\
&+\underbrace{\left(\frac{2(p+1)}{n-1}\underset{M}{\max} \ ({d_{x_0}.s_{\kappa_1}(d_{x_0}))}-(1+\underset{M}{\min} (d_{x_0}.s_{\kappa_2}(d_{x_0})))\right)}_{c_1}\\
& \times\int_{M} (|d \omega|^2 + |\delta \omega|^2)d\mu_g.\end{align*}
Now let $\mathcal{R}^2=\frac{\int_{\partial M}\lvert\nabla_\nu\w\rvert^2 d\mu_g}{ \int_M (\lvert d\w\rvert^2 +\lvert \delta\w\rvert ^2) d\mu_g}$. Dividing by $\int_M (\lvert d\w\rvert^2 +\lvert \delta\w\rvert ^2) d\mu_g,$ we obtain the second-degree polynomial in $\mathcal{R}$ : \[h_{\min}\mathcal{R}^2- 2 \mathbf{q_{k,p}} ^{\frac{1}{2}}r_{\max}\mathcal{R}-c_1\leq 0.\] 

   By calculating the discriminant and then the bigger root of this polynomial, we obtain that
 \begin{align*}
\mathcal{R}^2 
&\leq \frac{\left( \mathbf{q_{k,p}} ^{\frac{1}{2}} r_{\max} +  \sqrt{\mathbf{q_{k,p}} r_{\max}^2 + h_{\min} c_1} \right)^2}{h_{\min}^2} \\
&\leq \frac{2\mathbf{q_{k,p}} r^2_{\max}+2\mathbf{q_{k,p}} r^2_{\max}+2h_{\min} c_1}{h^2_{\min}}=\frac{4\mathbf{q_{k,p}} r^2_{\max}+2h_{\min}c_1}{h^2_{\min}}.
\end{align*}
Replacing the result in \eqref{eq:lamda<qsup}, we finally get \begin{equation}\label{eq:equationfinale}
\lambda_{k,p} \leq \frac{4\mathbf{q_{k,p}}^2 r^2_{\max}+2\mathbf{q_{k,p}} h_{\min} c_1}{h^2_{\min}}. \end{equation}
This completes the proof of \thref{thm:inegalité2}.

\end{proof} 

\begin{remark}
As $\mathbf{q_{k,p}}$ goes to infinity when $k$ goes to infinity, the upper bound in \eqref{eq:equationfinale} is positive regardless of the sign of $c_1$. We know that the limit of  ${r_{\max}s_{\kappa}(r_{\max})}$ when $ r_{\max} $ goes to zero is ${n-1}$ hence for $p>\frac{n}{2}-1$, the limit of $ c_1$ when $ r_{\max} $ goes to zero becomes $2(p+1)-n$ which is positive. The function $rs_{\kappa}(r )$ is constant if $\kappa = 0$, increasing on $[0,\infty)$ if $\kappa < 0$, and decreasing on $[0,\infty)$ if $\kappa> 0.$\\ Here are some values of $c_1$ depending on the variation of the sign of $\kappa_1$ and $\kappa_2$.
\begin{enumerate}
    \item[$\bullet$] \underline{$\kappa_1 = \kappa_2 = 0$ \ \text{and}\ $p \geq 1$: } $c_1=2(p+1)-n$. This includes the case when $M$ is a domain in $\mathbb{R}^n.$
    \item[$\bullet$] \underline{$\kappa_1\leq \kappa_2 \leq 0$: } $c_1=\frac{2(p+1)}{n-1}r_{\max}{s_{\kappa_1}(r_{\max})}-n.$
    \item[$\bullet$] \underline{$  0\leq\kappa_1\leq \kappa_2$:} $c_1=2(p+1)-(1+r_{\max}s_{\kappa_2}(r_{\max})).$
    \item[$\bullet$] \underline{$  \kappa_1\leq 0\leq \kappa_2$:} $c_1=\frac{2(p+1)}{n-1} r_{\max}{s_{\kappa_1}(r_{\max})}-(1+r_{\max}s_{\kappa_2}(r_{\max}))$.
\end{enumerate}
\end{remark}

\begin{remark}
   We tried to generalize the second inequality of \thref{thm:ineqScalaire2} by adapting the method used in the proof of \thref{thm:inegalité2}. Unfortunately, this approach yields a constant $c_2$ of the form
\[
c_2 = 
\left(1+\frac{2(p-1)}{n-1}\right)
\underset{r\in[0,r_{\max})}{\min}\; d_{x_0} \, s_{\kappa_2}(d_{x_0})
\;-\;
\left(1+\frac{2p}{n-1}\right)
\underset{r\in[0,r_{\max})}{\max}\; d_{x_0} \, s_{\kappa_1}(d_{x_0})
\] \end{remark}
which can sometimes be negative. However, the expected result would be as follows:
\begin{conjecture}
    Let $(M^n,g)$ be a compact Riemannian manifold, whose sectional curvature $\kappa_g$ satisfies $\kappa_1\leq \kappa_g \leq \kappa_2$. Assume that $M$ is star-shaped with respect to $x_0\in M$ then, for all $k\in\mathbb{N}$, $p\geq 1$, there exists a constant $c_2= c_2(p,n,\kappa_1,\kappa_2)>0$ such that:
\begin{equation*}
 \qquad    \lambda_{k,p}\geq \frac{c_2 q_{{m_k},p}}{r_{\max}+h_{\max}} . 
\end{equation*}
Here $m_k$ is the multiplicity of $\lambda_{k,p}$.
\end{conjecture}
This inequality could be of interest when $m_k$ is sufficiently large compared with $k.$

\bibliographystyle{abbrv}
\bibliography{references}

\end{document}